%% file: 2026Global_in_Time_arXiv.tex
\documentclass[hidelinks,onefignum,onetabnum]{siamart251216}

\usepackage{color}
\usepackage{cleveref}
\usepackage{mathtools} 

\usepackage{amsmath}
\usepackage{amssymb}

\newcommand{\dx}{\,{\mathrm d}x}

\newcommand{\SN}{{\mathcal{S}_{\mathcal{N}}}(\tau)}
\newcommand{\SLt}{{\mathcal{S}_{\mathcal{L}}\left(\frac{\tau}{2}\right)}}

\usepackage{lipsum}
\usepackage{amsfonts}
\usepackage{graphicx}
\usepackage{epstopdf}
\usepackage{algorithmic}
\usepackage{xcolor}
\usepackage{booktabs}
\usepackage{enumerate}

\ifpdf
  \DeclareGraphicsExtensions{.eps,.pdf,.png,.jpg}
\else
  \DeclareGraphicsExtensions{.eps}
\fi

\newsiamremark{hypothesis}{Hypothesis}
\crefname{hypothesis}{Hypothesis}{Hypotheses}
\newsiamthm{claim}{Claim}

\headers{Optimal \(\varepsilon\)-dependence}{R. Luo, C. Quan, Z. Tan, and Y. Wu}

\title{Stability and Error Estimates with Optimal $\varepsilon$-Dependence for Strang Splitting of the Allen--Cahn Equation
}

\author{
Rongfu Luo\thanks{School of Science and Engineering,  The Chinese University of Hong Kong (Shenzhen), Shenzhen, 518172, Guangdong, China (\email{rongfuluo@link.cuhk.edu.cn}).}
\and Chaoyu Quan\thanks{School of Science and Engineering, The Chinese University of Hong Kong (Shenzhen), Shenzhen
518172, China; Shenzhen International Center for Industrial and Applied Mathematics, Shenzhen
Research Institute of Big Data, Shenzhen, 518000, China (\email{quanchaoyu@cuhk.edu.cn}).}
\and Zhijun Tan\thanks{Sun Yat-sen University, School of Computer Science and Engineering, Guangzhou 510006, China (\email{tzhij@mail.sysu.edu.cn}).}
\and Yanyao Wu\thanks{Sun Yat-sen University, School of Computer Science and Engineering, Guangzhou 510006, China (\email{wuyy96@mail2.sysu.edu.cn}).}
}

\usepackage{amsopn}

\ifpdf
\hypersetup{
  pdftitle={$\varepsilon$-dependence of long-time Sobolev estimate},
  pdfauthor={R. Luo, C. Quan, Z. Tan, and Y. Wu}
}
\fi

\begin{document}

\maketitle

\begin{abstract}
In this work, long-time Sobolev stability and error estimates with optimal $\varepsilon$-dependence are established for the Strang splitting method applied to the Allen--Cahn equation with periodic boundary conditions. By combining the modified energy dissipation law, and H\"older--Gagliardo--Nirenberg product estimates, the numerical solution is proven to admit uniform-in-time $H^r$-seminorm bounds $\mathcal{O}(\varepsilon^{-r}+1)$, consistent with the corresponding Sobolev bounds of the exact solution. A second-order error estimate $\mathcal{O}(\varepsilon^2\tau^2)$ is further established on any prescribed finite time interval, exhibiting quadratic dependence on $\varepsilon$. Numerical experiments are presented to verify the optimal $\varepsilon$-dependence of both the stability and error estimates.
\end{abstract}

\begin{keywords}
Allen--Cahn equation, Strang splitting, modified energy dissipation, long-time stability, error estimates
\end{keywords}

\begin{MSCcodes}
65M12, 65M15
\end{MSCcodes}

\section{Introduction}
The Allen--Cahn (AC) equation \cite{allen1979microscopic}
\begin{equation}
    \left\{
    \begin{aligned}
        &\partial_t u =\varepsilon^2 \Delta u-f(u), && (t,x)\in (0,T)\times \Omega, \\
        & u(0,x) = u^0(x), &&\mbox{in } \Omega, \\
    \end{aligned}
    \right.
    \label{eq:AC}
\end{equation}
with Ginzburg--Landau energy functional
\begin{align}
    E(u) = \int_{\Omega} \left(\frac{\varepsilon^2}{2}|\nabla u|^2 + F(u)\right)  \dx,
\end{align}
is a model used to describe phase transitions and interfacial evolution.
It is assumed that the spatial domain $\Omega$ is a bounded open set in ${\mathbb{R}^d}~(d \le 3)$ with periodic boundary condition, $f(u)=u^3-u$ is taken as the derivative of the energy function $F(u)=\frac{1}{4}(u^2-1)^2$, and the parameter $0<\varepsilon\leq 1$ characterizes the thickness of the diffuse interface. 
Due to the double-well structure of \(F(u)\), the solution \(u\) is determined by time-dependent interfaces \(\Gamma(t):=\{x\in\Omega:u(x,t)=0\}\) as \(\varepsilon\to0\), which separate the two phases. The solution evolves between the two regions within the length \(O(\varepsilon)\). 

At the continuous level, the long-time Sobolev estimates of the exact solution for the Allen--Cahn equation have been well studied, as well as the explicit dependence on \(\varepsilon\). 
It is known that as \(\varepsilon\to0\), the diffuse interface generated by the Allen--Cahn equation converges to a sharp interface evolving by mean curvature flow \cite{evans1992,de1995geometrical}. The spectral stability of the linearized Allen--Cahn operator for sufficiently regular interfaces is given in  \cite{chen1994spectrum,de1995geometrical}, which has played an important role in obtaining error estimates with polynomial dependence on \(\varepsilon^{-1}\). These continuous-level results motivate the development of numerical methods that preserve long-time Sobolev stability and maintain an explicit dependence on \(\varepsilon\) consistently.

At the discrete level, as pointed out by Blowey and Elliott, a direct Gr\"onwall inequality argument leads to an error estimate with constant \(e^{T/\varepsilon^2}\) that could be extremely large when \(\varepsilon\) is small \cite[equation 4.10]{blowey1993curvature}. 
To overcome this difficulty, Feng and Prohl employed the spectral estimate for the linearized AC operator established in \cite{chen1994spectrum}, together with its discrete counterpart, to derive error estimates with only polynomial dependence on \(\varepsilon^{-1}\) for the backward Euler time discretization and its fully discrete realization using conforming piecewise linear finite elements in space \cite{feng2003numerical}.
This technique was also used to remove the exponential dependence on $\varepsilon^{-1}$, including the implicit Euler method for the nonisothermal AC equations \cite{feng2004b}, stabilized semi-implicit schemes \cite{yang2009} and a posteriori error estimates \cite{bartels2011robust,chrysafinos2020posteriori} for the AC equation. These works concern the implicit and semi-implicit Euler methods and derive error estimates in the form of $O\left(\left(\tau+h^r\right) \varepsilon^{-\sigma}\right)$, where $\tau$ is the time step, $h$ is the spacing size, and $\sigma$ is some positive number. 
Recently, Akrivis and Li derive in \cite{akrivis2022error} fully discrete error estimates for high-order backward differentiation formula (BDF) finite element methods of the AC equation with
explicit dependence on $\varepsilon$. Using the spectral estimate for the linearized AC operator,
they obtained an \(L^2\)-norm error bound
\(
O\bigl(
\tau^q\varepsilon^{-q-\frac12}
+h^r\varepsilon^{-r-\frac12}
+e^{-c/\varepsilon}
\bigr),
\)
where \(1\le q\le5\) denotes the BDF order and \(r-1\) denotes the polynomial degree of the finite element space. 
In the literature reviewed above, the Allen--Cahn equation is typically written in the rescaled form
\[
\partial_t u = \Delta u - \frac{1}{\varepsilon^2} f(u).
\]
For the error estimates obtained in this rescaled setting, the $\varepsilon^{-\sigma}$-dependence appears to be an artifact of the analytical techniques employed, which raises a natural question how to achieve the optimal $\varepsilon$-dependence error estimates that is consistent with numerical observation.


In this work, we take the second-order Strang splitting method \cite{strang1968construction} as an example, to establish the long-time Sobolev stability and error estimates with optimal \(\varepsilon\)-dependence for the AC equation, which improves the results in \cite{li2022stabilityI}. The Strang splitting is written as
\begin{align} \label{eq:strang}
    u^{n+1} = \mathcal S(\tau) u^n = \SLt \SN\SLt u^n,
\end{align}
where $0<\tau \leq 1$ is the time step, $\SLt =\exp\left(\frac{\tau}{2}\varepsilon^2\Delta\right)$ is the linear propagator and $\mathcal{S}_{\mathcal{N}}(\tau)$ is the solution operator for the nonlinear part 
    \begin{align} \label{eq:SN(v)}
        \partial_t u = u-u^3.
    \end{align}
Based on the unconditional energy dissipation law established in \cite{li2022stabilityI,quan2026unconditional}, we use a H\"older--Gagliardo--Nirenberg product inequality to derive the following uniform-in-time bound of numerical solutions:
\[
    \sup_{n\geq0}\|u^n\|_{H^r}
    \leq
    C_{r,u^0}\bigl(\varepsilon^{-r}+1\bigr),
\]
where \(C_{r,u^0}\) depends only on \(\Omega\), \(d\), \(r\), and the initial data $u^0$. It is shown that this bound is consistent with the bound of exact solution derived from PDE analysis. Further, we employ the iterated Lie commutators \cite{lubich2008splitting, jahnke2000error} to derive local consistency errors in the \(H^r\)-norm and combine it with the preceding uniform stability bounds to derive
\[
    \sup_{n\tau\leq T}
    \|u^n-u(t_n)\|_{H^r}
    \leq
    C_{r,T,u^0}\varepsilon^{2}\tau^2,
\]
where $u^n$ and $u(t_n)$ are the numerical and exact solutions respectively, \(C_{r,T,u^0}\) depends only on $\Omega$, $d$, $r$, $T$, and $\|u^0\|_{H^{r+4}\cap W^{4,\infty}}$. 
Several numerical experiments are conducted to verify the predicted second-order temporal accuracy, the maximum principle, the energy dissipation property, 
the optimal dependence on \(\varepsilon\) of the uniform-in-time Sobolev bounds and the error estimate.

The remainder of the paper is organized as follows. In section \ref{sec:bound continuous}, we introduce some preliminaries and show the long-time Sobolev bounds of the exact solution of the AC solution. Section \ref{sec:bound} proves the long-time \(H^r\)-norm bound of the numerical solution \(u^n\) with dependency on $\varepsilon^{-1}$ consistent with the continuous level. Section \ref{sec:error} establishes the consistency and global error analysis, which proves second-order convergence with optimal dependence on \(\varepsilon\). In section \ref{sec:numerical}, we present several numerical experiments to validate the long-time Sobolev stability and the optimal \(\varepsilon\)-dependent error estimates. Finally, we conclude the paper in section \ref{sec:conclusion}.

\section{Long-time Sobolev bounds of exact solution}
\label{sec:bound continuous}

We shall use the following definition of the $H^r$ norm:
\begin{equation}\label{def:Hknorm}
    \left\| u \right\|_{H^r}^2
    \coloneqq
    \sum_{|\alpha|\le r} \left\|D^{\alpha}u\right\|_{L^2}^2,
\end{equation}
where \(\alpha=(i_1,i_2,\cdots,i_d)\in\mathbb N^d\) is a multi-index, \(|\alpha| \coloneqq \sum_{j=1}^d i_j\) and
\begin{equation}
\begin{aligned}
    D^{\alpha} u(t,x)
    :=
    \frac{\partial^{|\alpha|}u(t,x)}
    {\partial x_1^{i_1}\partial x_2^{i_2}\cdots\partial x_d^{i_d}} .
\end{aligned}
\label{eq:multiindex}
\end{equation}
Let \(-\Delta\) denote the Laplace operator under periodic boundary conditions. We define
\begin{equation}\label{def:Lambda}
    \Lambda:=(-\Delta)^{1/2}
\end{equation}
as the square root of \(-\Delta\).
More precisely, let \(\{\phi_j\}_{j\ge0}\) be an orthonormal basis of \(L^2(\Omega)\) consisting of eigenfunctions of \(-\Delta\), and let \(\{\lambda_j\}_{j\ge0}\) be the corresponding eigenvalues, namely
\begin{equation}\label{def:eigenvalue}
        -\Delta\phi_j=\lambda_j\phi_j,
    \qquad
    0\leq \lambda_0\le \lambda_1\le \lambda_2\le\cdots .
\end{equation}
As a consequence, when
\[
    v=\sum_{j\ge0}v_j\phi_j,
\]
we define the seminorm
\begin{equation}\label{def:seminorm}
    |v|_r
    :=
    \|\Lambda^r v\|_{L^2}
    =
    \left(
    \sum_{j\ge0}\lambda_j^r |v_j|^2
    \right)^{1/2},
    \qquad r\ge0
\end{equation}
and the norm
\begin{equation}\label{def:norm}
    \|v\|_{r,\Lambda}
    :=
    \left(
    \|v\|_{L^2}^2+|v|_r^2
    \right)^{1/2}.
\end{equation}

For Strang splitting method \eqref{eq:strang}, $\SLt =\exp\left(\frac{\tau}{2}\varepsilon^2\Delta\right)$ is the linear propagator and $\mathcal{S}_{\mathcal{N}}$ is the solution operator for the nonlinear part 
\begin{align} \label{eq:SN(v)}
    \partial_t u = u-u^3.
\end{align}
 Here, $\SN v$ can be calculated explicitly by solving an ordinary differential equation if the initial condition $v$ is given.
 The Strang splitting scheme defines three intermediate steps:
\begin{align}
u^{n,(1)}  &= e^{\frac{\tau}{2}\varepsilon^{2}\Delta} u^n, \label{eq:step1_1}\\
u^{n,(2)}  &= \frac{e^{\tau}u^{n,(1)}}{\sqrt{1+(e^{2\tau}-1)(u^{n,(1)})^{2}}} , \label{eq:step2_1}\\
u^{n+1}    &= e^{\frac{\tau}{2}\varepsilon^{2}\Delta} u^{n,(2)}. \label{eq:step3_1}
\end{align}

\begin{lemma}[Equivalence of Sobolev norms]
\label{lem:sobolev_equivalence}
Let \(r\ge0\) be an integer. With the definition \eqref{def:Lambda} of \(\Lambda\), there exist constants \(c_r,\tilde c_r>0\), depending only on \(\Omega\), \(d\), and \(r\), such that
\begin{equation}
    c_r\|u\|_{H^r}
    \le
    \|u\|_{r,\Lambda}
    \le
    \tilde c_r\|u\|_{H^r},
    \qquad
    \forall u\in H^r(\Omega).
    \label{eq:spectral_sobolev_equivalence}
\end{equation}
\end{lemma}

\begin{proof}
The result follows from the spectral characterization of Sobolev norms, together with the standard elliptic regularity estimate for the corresponding Laplace operator. We omit the details.
\end{proof}

Now we establish the following Hölder--Gagliardo--Nirenberg product
estimate, which will be used repeatedly in the subsequent analysis.

\begin{lemma}[H\"older--Gagliardo--Nirenberg product estimate]
\label{lem:HGN}
Let $k\geq 1$ be an integer. Suppose that
\(v_i\in H^k(\Omega)\cap L^\infty(\Omega)\),
\(1\leq i\leq q\), and \(
\max_{1\leq i\leq q}\|v_i\|_{L^\infty}\leq 1.
\)
Let $\beta_i\in\mathbb{N}^d$, $1\leq i\leq q$, with
\(\beta=\sum_{i=1}^q\beta_i\) and 
\(m:=|\beta|\leq k\).
Denote \(
p_m:=\frac{2k}{m},
\)
and $p_0=\infty$. 
For $0\leq m\leq k$, we then have
\begin{equation}
\left\|
\prod_{i=1}^q D^{\beta_i}v_i
\right\|_{L^{p_m}}
\leq
C_{\beta}\,\max_{1\leq i\leq q}\|v_i\|_{H^k}^{\frac m k}.
\label{eq:product l 2r/m}
\end{equation}
Moreover, let $r\in\mathbb{N}$ satisfy
\(k=r+m\geq 1\). The following $H^r$-norm bound holds
\begin{equation}
\left\|
\prod_{i=1}^q D^{\beta_i}v_i
\right\|_{H^r}
\leq
\tilde C_{\beta}\,\max_{1\leq i\leq q}\|v_i\|_{H^k}.
\label{eq:product estimate hr}
\end{equation}
Here, \(C_{\beta},\,\tilde C_{\beta}>0\) depend only on
$\Omega$, $d$, $k$, $q$, and the multi-indices
$\beta_1,\ldots,\beta_q$.
\end{lemma}

\begin{proof}
Let \(M := \max_{1\leq i\leq q}\|v_i\|_{H^k}\). The case $m=0$ follows immediately from the $L^\infty$-norm bounds. We next consider the case $m>0$. Define
\[
p_i:=
\begin{cases}
\dfrac{2k}{|\beta_i|}, & \text{if }|\beta_i|>0,\\[1mm]
\infty, & \text{if }|\beta_i|=0,
\end{cases}
\qquad i=1,\ldots,q.
\]
The Gagliardo--Nirenberg interpolation inequality
\cite{nirenberg1959elliptic} gives, for $|\beta_i|>0$,
\[
\|D^{\beta_i}v_i\|_{L^{p_i}}
\leq
C_{k,q,\beta}^{(1)}
\|v_i\|_{L^\infty}^{1-|\beta_i|/k}
\|v_i\|_{H^k}^{|\beta_i|/k},
\]
where \(C_{k,q,\beta}^{(1)}\) depends only on
\(\Omega,\,d,\,k,\,q\), and \(\beta_1,\,\ldots,\,\beta_q\).
Since
\(
\sum_{i=1}^q\frac{1}{p_i}
=
\frac{1}{p_m},
\)
using the H\"older's inequality and this  Gagliardo--Nirenberg interpolation inequality yields
\begin{equation}
\left\|
\prod_{i=1}^q D^{\beta_i}v_i
\right\|_{L^{p_m}}
\leq
C_{k,q,\beta}^{(2)}
\prod_{\{i:\,|\beta_i|=0\}}\|v_i\|_{L^\infty}
\prod_{\{i:\,|\beta_i|>0\}}
\|v_i\|_{L^\infty}^{1-|\beta_i|/k}
\|v_i\|_{H^k}^{|\beta_i|/k},
\label{eq:product estimate}
\end{equation}
where \(C_{k,q,\beta}^{(2)}\) depends only on
\(\Omega,\,d,\,k,\,q\), and \(\beta_1,\ldots,\beta_q\).
Since \(\max_{1\leq i\leq q}\|v_i\|_{L^\infty}\leq 1\), we can obtain \eqref{eq:product l 2r/m} directly from \eqref{eq:product estimate}:
\[
\begin{aligned}
\left\|
\prod_{i=1}^qD^{\beta_i}v_i
\right\|_{L^{p_m}}
&\leq
C_{\beta}\,
M^{\sum_{i=1}^q|\beta_i|/k}
\leq
C_{\beta}\,M^{m/k}.
\end{aligned}
\]
For the second estimate, let \(\alpha\in\mathbb{N}^d\), with
\(n:=|\alpha|\leq r\). By the Leibniz rule,
\[
D^\alpha\left(\prod_{i=1}^qD^{\beta_i}v_i\right)
=
\sum_{\substack{\alpha_1,\ldots,\alpha_q\in\mathbb{N}^d\\
\alpha_1+\cdots+\alpha_q=\alpha}}\frac{\alpha!}{\alpha_1!\cdots\alpha_q!}
\prod_{i=1}^qD^{\beta_i+\alpha_i}v_i.
\]
For each term, set \(s:=m+n\leq k\). If \(s=0\), then
\[
\left\|\prod_{i=1}^qv_i\right\|_{L^2}
\leq
\|v_1\|_{L^2}\prod_{i=2}^q\|v_i\|_{L^\infty}
\leq M.
\]
If \(s>0\), applying \eqref{eq:product estimate} with \(k\) replaced by \(s\) and
\(\beta_i\) replaced by \(\beta_i+\alpha_i\), and using
\(\|v_i\|_{H^s}\leq\|v_i\|_{H^k}\leq M\), we obtain
\[
\left\|
\prod_{i=1}^qD^{\beta_i+\alpha_i}v_i
\right\|_{L^2}
\leq
C_{k,q,\alpha,\beta}\,M^{\sum_{i=1}^q|\beta_i+\alpha_i|/s}
=
C_{k,q,\alpha,\beta}\,M,
\]
where \(C_{k,q,\alpha,\beta}\) depends on
\(\Omega,\,d,\,k,\,q,\,\alpha\), and \(\beta\).
Summing over \(|\alpha|\leq r\) proves
\eqref{eq:product estimate hr}.
\end{proof}

\begin{lemma}[Long-time bounds of the Allen--Cahn solution]
\label{lem:longtime-bound}
Assume that the initial data \(u^0\) satisfies
\[
    \|u^0\|_{L^\infty}\le 1,
\]
and that, for some given integer \(m\ge1\),
\[
    |u^0|_m
    < \infty.
\]
Then there exists a positive constant $C_m>0$ depending only on \(\Omega\), \(d\), \(m\), and \(|u^0|_m\) such that the solution \(u(t)\) of the Allen--Cahn equation \eqref{eq:AC} satisfies
\begin{equation}
    |u(t)|_m\leq
     C_m (\varepsilon^{-m}+1),
    \quad
    \forall\, t\ge0,~ m\ge1.
    \label{eq:exact_solution_Lambda_bound}
\end{equation}
\end{lemma}

\begin{proof}
The proof follows from standard partial differential equation analysis and is provided in Appendix \ref{append:1} for interested researchers.
\end{proof}

The preceding lemma shows that the Sobolev seminorm of the exact Allen--Cahn solution remains uniformly bounded in time. We next establish the corresponding estimate for the nonlinear ordinary differential equation
\begin{equation}
\left\{
\begin{aligned}
&\partial_t u  = u - u^{3}, \\
&u(x,0)  = u^0.
\end{aligned}
\right.
\label{eq:ode}
\end{equation}

\begin{lemma}
\label{ode}
Assume that \(u^0\in H^r(\Omega)\) satisfies the periodic boundary condition, and
\(
    \|u^0\|_{L^\infty}\le 1.
\)
For each integer \(r\ge0\), there exist two constants \(C_r,~\widetilde C_r>0\), depending only on \(\Omega\), \(d\), and \(r\), such that
\begin{equation}
    \left|
    \mathcal S_{\mathcal N}(t)u^0
    \right|_r
    \le
     C_r
    \left(
        |u^0|_r+1
    \right),
    \quad \forall 0<t\le1 
    \label{eq:SN_Lambda_r_bound}
\end{equation}
and
\begin{equation}
    \left|
    \left(\mathcal S_{\mathcal N}(t)u^0\right)^3
    \right|_r
    \le
    \widetilde C_r
    \left(
        |u^0|_r+1
    \right),
    \quad \forall 0<t\le1 .
    \label{eq:SN_cube_Lambda_r_bound}
\end{equation}
\end{lemma}

\begin{proof}
The solution of the nonlinear subproblem \eqref{eq:ode} is given explicitly by
\[
    \mathcal S_{\mathcal N}(t)u^0
    =
    g_t(u^0),
\]
where
\[
    g_t(z)
    :=
    \frac{e^t z}{\sqrt{1+(e^{2t}-1)z^2}},
    \qquad z\in\mathbb R,\quad t\in[0,1].
\]
Since
\[
    1+(e^{2t}-1)z^2\ge 1,
    \qquad t\in[0,1],
\]
it follows that, for every integer \(r\ge1\),
\begin{equation}
    \max_{1\le k\le r}
    \sup_{t\in[0,1]}
    \sup_{|z|\le1}
    \left|g_t^{(k)}(z)\right|
    \le C_{g,r},
    \label{eq:Ft_derivative_bound_Lambda}
\end{equation}
where \(C_{g,r}>0\) depends only on \(r\).

We first derive an \(H^r\)-estimate for \(g_t(u^0)\). For \(r=0\), the desired estimate \eqref{eq:SN_Lambda_r_bound} follows directly from the explicit formula:
\[
    \left\|\mathcal S_{\mathcal N}(t)u^0\right\|_{L^2}
    =
    \left\|
    \frac{e^t u^0}{\sqrt{1+(e^{2t}-1)(u^0)^2}}
    \right\|_{L^2}
    \le
    e^t\|u^0\|_{L^2}
    \le
    e\|u^0\|_{L^2}.
\] For $r\geq 1$, let \(\beta\) be a multi-index with \(|\beta|=r\). By the multidimensional Fa\`a di Bruno formula,
\begin{equation}
\begin{aligned}
    D^\beta g_t(u^0)
    &=
    g'(u^0)D^\beta u^0  
    +
    \sum_{k=2}^{r}
    g^{(k)}(u^0)
    \sum_{\substack{\beta_1+\cdots+\beta_k=\beta\\
                    |\beta_i|\ge 1}}
    c_{\boldsymbol\beta}
    \prod_{i=1}^{k}D^{\beta_i}u^0 ,
\end{aligned}
\label{eq:Faa_di_Bruno_SN_Lambda}
\end{equation}
where \(c_{\boldsymbol\beta}\) depends only on multi-indices $\beta_1,\ldots,\beta_k$. Using \eqref{eq:Ft_derivative_bound_Lambda} and the assumption \(\|u^0\|_{L^\infty}\le1\), we obtain
\begin{equation}
\begin{aligned}
    \left\|D^\beta g_t(u^0)\right\|_{L^2}
    &\le
    C_{\Omega,d,r}^{(1)}
    \left\|D^\beta u^0\right\|_{L^2} 
    +
    C_{\Omega,d,r}^{(1)}
    \sum_{k=2}^{r}
    \sum_{\substack{\beta_1+\cdots+\beta_k=\beta\\
                    |\beta_i|\ge 1}}
    \left\|
    \prod_{i=1}^{k}D^{\beta_i}u^0
    \right\|_{L^2},
\end{aligned}
\label{eq:r_order_SN_pre_Lambda}
\end{equation}
where \(C_{\Omega,d,r}^{(1)}>0\) depends only on \(\Omega\), \(d\), and \(r\). 
Since $\|u^0\|_{L^\infty}\leq1$, for each decomposition
\(
\beta_1+\cdots+\beta_k=\beta,
\, |\beta_i|\geq1,
\)
the estimate \eqref{eq:product l 2r/m} in Lemma~\ref{lem:HGN} (by taking $m=k$) gives
\begin{equation}
\left\|\prod_{i=1}^kD^{\beta_i}u^0\right\|_{L^2}
\leq C_{\beta}\|u^0\|_{H^r}.
\label{eq:product_SN_bound_Lambda}
\end{equation}
Substituting \eqref{eq:product_SN_bound_Lambda} into \eqref{eq:r_order_SN_pre_Lambda}, we arrive at
\begin{equation}
    \left\|D^\beta g_t(u^0)\right\|_{L^2}
    \le
    C_{\Omega,d,r}^{(2)}
    \left\|u^0\right\|_{H^r},
    \quad \forall |\beta|=r,\quad 0<t\le1,
    \label{eq:D_beta_composition_bound_Lambda}
\end{equation}
where \(C_{\Omega,d,r}^{(2)}>0\) depends only on \(\Omega\), \(d\), and \(r\), but not on \(t\).
By Lemma~\ref{lem:sobolev_equivalence} and the fact that \(\|u^0\|_{L^\infty}\le1\), we then have
\[
    |g_t(u^0)|_r
    \le
    C_{\Omega,d,r}^{(3)}
    \|u^0\|_{H^r}\le C_{\Omega,d,r}^{(4)}
    \left(
        |u^0|_r+\|u^0\|_{L^2} 
    \right)\le C_r \left(
        |u^0|_r+1
    \right),
    \qquad 0<t\le1,
\]
where \(C_{\Omega,d,r}^{(3)},~C_{\Omega,d,r}^{(4)},~C_r>0\) depend only on \(\Omega\), \(d\), and \(r\), but not on \(t\).
This proves \eqref{eq:SN_Lambda_r_bound} for \(r\ge1\).

It remains to bound the cubic term in \eqref{eq:SN_cube_Lambda_r_bound}. Define
\[
    h_t(z):=\left(g_t(z)\right)^3.
\]
Since \(g^{(k)}(z)\) is uniformly bounded for \(0<t\le1\), \(|z|\le1\), and \(0\le k\le r\), the derivatives of \(h_t\) are also uniformly bounded, namely
\[
    \max_{0\le k\le r}
    \sup_{t\in[0,1]}
    \sup_{|z|\le1}
    |h_t^{(k)}(z)|
    \le
    C_{h,r}.
\]
Applying the preceding composition estimate to \(h_t(u^0)\), we obtain a similar result
\[
    |h_t(u^0)|_r
    \le
    \widetilde C_r
    \left(
        |u^0|_r+1
    \right),
    \qquad 0<t\le1 .
\]
This proves \eqref{eq:SN_cube_Lambda_r_bound}, and the proof is complete.
\end{proof}

\section{Long-time Sobolev bounds of numerical solution}
\label{sec:bound}

In \cite{li2022stabilityI}, the authors prove that the Strang splitting is unconditionally energy stable when solving the Allen--Cahn equation, which is useful for global-in-time stability analysis. 
\begin{lemma}[Energy dissipation law \cite{li2022stabilityI}]\label{energy dissipation law}
Assume $u^{0} \in H^{1}(\Omega)$ satisfies the periodic boundary condition, and
\(
    \|u^0\|_{L^\infty}\le 1.
\) For any $\tau > 0$ and any $n \geq 0$, the Strang splitting method \eqref{eq:strang} satisfies the following energy dissipation law:
\begin{equation}
\widetilde{E}^{n+1} \leq \widetilde{E}^{n},
\end{equation}
where
\begin{equation}\label{eq:energy}
\begin{aligned}
\widetilde{E}^{n} 
&= \frac{1}{2\tau} \left\langle \left(e^{-\varepsilon^{2}\tau\Delta} -\mathcal{I}\right) u^{n,(1)}, u^{n,(1)} \right\rangle + \int_{\Omega} \widetilde{F}\left(u^{n,(1)}\right)\,\mathrm dx, \\
\widetilde{F}\left(u^{n,(1)}\right) 
&= \frac{1}{4} + \frac{1}{2\tau} \left(u^{n,(1)}\right)^{2} - \frac{e^{\tau}}{\tau \left(e^{2\tau} -1\right)} \left( \sqrt{  1 + \left(e^{2\tau} - 1\right) \left(u^{n,(1)}\right)^{2}  } - 1 \right),
\end{aligned}
\end{equation}
and $u^{n,(1)}$ is defined as \eqref{eq:step1_1}.
\end{lemma}
\begin{proof}
     The case of periodic boundary condition has been proved in \cite{li2022stabilityI}. 
\end{proof}

With this lemma, we are ready to state and prove the following global-in-time \(H^r\)-seminorm bound of the numerical solutions in the Strang splitting method, with the dependency of the upper bound on $\varepsilon$ and $u^0$ given explicitly.

\begin{theorem}[Long-time \(H^r\)-seminorm bound of \(u^n\)]
\label{thm:uniform bound}
Assume that \(0<\varepsilon\le 1\), \(\tau\le 1\), and the initial data
\(u^0\) satisfies
\(
    \|u^0\|_{L^\infty}\le 1
\)
and for some integer \(m\ge1\),
\(
    |u^0|_m< \infty.
\)
Let 
\begin{equation}
    R_r(\varepsilon)
    :=
    \sup_{n\ge0}|u^n|_r.
\end{equation}
Then the numerical solution \(u^n\) generated by the Strang splitting method \eqref{eq:strang} satisfies the global-in-time estimate
\begin{equation}
\begin{aligned}
    R_1(\varepsilon)
    &\le
    K_{1,u^0}
    \left(
        \varepsilon^{-1}+1
    \right),\\
    R_{r+1}(\varepsilon)
    &\le
    \max\left\{|u^0|_{r+1},
    C_{\Omega,d,r}
    \frac{R_r(\varepsilon)+1}{\varepsilon}\right\},\qquad 1\leq r\leq m-1,
    \end{aligned}
    \label{eq:R1_explicit_recursive}
\end{equation}
where $K_{1,u^0}
    :=
    \max\left\{
        |u^0|_1,\,
        C_{\Omega,d,1}\left(\sqrt{C_E}+1\right)
    \right\}$, $C_{\Omega,d,r}>0$ are constants depending only on $(\Omega,d,r)$, and  $C_E \coloneqq 2 E^0 + 4 C_{\Omega}$ with  initial energy $E^0 = E(u^0)$ and  a constant $C_{\Omega}>0$ depending only on $\Omega$.
    In particular, \eqref{eq:R1_explicit_recursive} implies
    \begin{equation}
    R_r(\varepsilon)
    \le
    K_{r,u^0}
    \left(
        \varepsilon^{-r}+1
    \right),
    \qquad 1\leq r\leq m,
    \label{eq:explicit_Lambda_eps_bound}
\end{equation}
where the constants \(K_{r,u^0}>0\) depend only on \(\Omega\), \(d\), \(r\), the initial data $u^0$, and  the initial energy $E^0$.
\end{theorem}

\begin{proof}
We prove the theorem by mathematical induction on the derivative
order \(r\). Throughout the proof, \(C_{\Omega,d,r}>0\) denotes a generic
constant depending only on \(\Omega\), \(d\), and \(r\), but independent
of \(n\), \(T\), \(\tau\), and \(\varepsilon\). The value of
\(C_{\Omega,d,r}\) may change from line to line.
The discrete maximum principle is straightforward: from
\(\|u^0\|_{L^\infty}\le 1\), we have
\begin{equation}
    \|u^n\|_{L^\infty}\le 1,
    \qquad
    \|u^{n,(1)}\|_{L^\infty}\le 1,
    \qquad
    \|u^{n,(2)}\|_{L^\infty}\le 1,
    \qquad n\ge0.
    \label{eq:maximum_bound_Lambda_theorem}
\end{equation}

We now prove the first-order estimate, i.e., the case of $r=1$. According to
Lemma~\ref{energy dissipation law}, the following inequality holds:
\[
\frac{1}{2\tau}
\left\langle
\left(e^{-\varepsilon^2\tau\Delta}-\mathcal I\right)
 u^{n,(1)},u^{n,(1)}
\right\rangle
+
\int_{\Omega}\widetilde F(u^{n,(1)})\,\mathrm dx
\le
\widetilde E^0 .
\]
Since \(\|u^{n,(1)}\|_{L^\infty}\le 1\), there exists a constant
\(C_\Omega>0\), independent of \(n\), \(T\), and \(\tau\), such
that
\[
    \sup_{n\geq 0} \left|
    \int_{\Omega}\widetilde F(u^{n,(1)})\,\mathrm dx
    \right|
    \le C_\Omega.
\]
As a consequence, we have
\begin{equation}
    \frac{1}{\tau}
    \left\langle
    \left(e^{-\varepsilon^2\tau\Delta}-\mathcal I\right)
    u^{n,(1)},u^{n,(1)}
    \right\rangle
    \le 2(\widetilde E^0+C_\Omega)\leq 2(E^0+2 C_\Omega) =:C_E .
    \label{eq:energy_first_order_Lambda}
\end{equation}
Using the spectral expansion of \(-\Delta\), we write
\[
    u^{n,(1)}
    =
    \sum_j \widehat{u_j^{n,(1)}}\phi_j,
    \qquad
    -\Delta\phi_j=\lambda_j\phi_j,
    \qquad
    \lambda_j\ge0.
\]
It follows from \eqref{eq:energy_first_order_Lambda} that
\[
    \frac{1}{\tau}
    \sum_j
    \left(e^{\varepsilon^2\tau\lambda_j}-1\right)
    |\widehat{u_j^{n,(1)}}|^2
    \le C_E ,
\]
which yields
\begin{equation}
\begin{aligned}
    |u^{n,(1)}|_1^2
    &=
    \sum_j
    \lambda_j
    |\widehat{u_j^{n,(1)}}|^2      \le
    \sum_{\lambda_j>0}
    \frac{
        e^{\varepsilon^2\tau\lambda_j}-1
    }{
        \varepsilon^2\tau\lambda_j
    }
    \lambda_j
    |\widehat{u_j^{n,(1)}}|^2       \le
    \frac{C_E}{\varepsilon^2}.
\end{aligned}
\label{eq:first_order_Lambda_intermediate}
\end{equation}
Here, we use the fact $\frac{e^x-1}{x}\ge1,~\forall x>0$.
By the contraction property of $\mathcal S_\mathcal{L}(\tau/2)$, Lemma~\ref{ode}, and \eqref{eq:first_order_Lambda_intermediate},
we obtain that for any $n\geq 0$,
\[
\begin{aligned}
    |u^{n+1}|_1
    &=
    \left|
    \SLt \mathcal S_{\mathcal N}(\tau)u^{n,(1)}
    \right|_1                                      \le
    \left|
    \mathcal S_{\mathcal N}(\tau)u^{n,(1)}
    \right|_1                                   \le
     C_1
    \left(
        |u^{n,(1)}|_1+1
    \right) \\    &\le
   C_{\Omega,d,1}\left(\sqrt{C_E}+1\right)
    \left(
        \varepsilon^{-1}+1
    \right).
\end{aligned}
\]
Therefore,
\begin{equation}
    R_1(\varepsilon)
    =
    \sup_{n\ge0}|u^n|_1
    \le
    K_{1,u^0}
    \left(
        \varepsilon^{-1}+1
    \right).
    \label{eq:R1_Lambda_bound}
\end{equation}
This proves the base case of the induction.

Next, we consider the case of $r+1$ with $r\ge 1$.
For the first linear half-step $\SLt$, we have
\begin{equation}\label{eq:firsthalf}
    {u^{n,(1)}-u^n}
    =
    \left(
    \mathcal I-e^{-\frac{\tau}{2}\varepsilon^2\Delta}
    \right)u^{n,(1)} .
\end{equation}
Applying \(\Lambda^{r+1}\) to both sides of \eqref{eq:firsthalf} and
taking the \(L^2\)-inner product with
\(2\Lambda^{r+1}u^{n,(1)}\), 
we obtain
\begin{equation}
\begin{aligned}
    |u^{n,(1)}|_{r+1}^2
    -
    |u^n|_{r+1}^2
    +
    |u^{n,(1)}-u^n|_{r+1}^2
    \le
    -\tau\varepsilon^2 |u^{n,(1)}|_{r+2}^2,
\end{aligned}
\label{eq:first_linear_Lambda_r}
\end{equation}
where we use the inequality $1-e^x\le -x,~\forall x\ge 0$.
For the nonlinear step $\mathcal{S}_{\mathcal{N}}(\tau)$, let
\[
    w(s):=\mathcal S_{\mathcal N}(s)u^{n,(1)},
    \qquad 0<s\le\tau.
\]
Then
\begin{equation}\label{eq:nonlinearstep}
    u^{n,(2)}-u^{n,(1)}
    =
    \int_0^\tau
    \left(
        w(s)-w^3(s)
    \right)\,\mathrm ds .
\end{equation}
Applying \(\Lambda^{r+1}\) to both sides of \eqref{eq:nonlinearstep} and taking the \(L^2\)-inner product with \(2\Lambda^{r+1}u^{n,(2)}\), we have
\begin{equation}
\begin{aligned}
& |u^{n,(2)}|_{r+1}^2
  -
  |u^{n,(1)}|_{r+1}^2
  +
  |u^{n,(2)}-u^{n,(1)}|_{r+1}^2                                      \\
={}&
  2\left\langle\Lambda^{r+1}
  \int_0^\tau
  \left(
      w(s)-w^3(s)
  \right)\,\mathrm ds,
  \Lambda^{r+1}u^{n,(2)}
  \right\rangle                                      \\
={}&
  2\left\langle
  \Lambda^r
  \int_0^\tau
  \left(
      w(s)-w^3(s)
  \right)\,\mathrm ds,
  \Lambda^{r+2}u^{n,(2)}
  \right\rangle                                      \\
\le{}&
  2|u^{n,(2)}|_{r+2}
  \int_0^\tau
  \left(
      |w(s)|_r+|w^3(s)|_r
  \right)\,\mathrm{d}s.
\end{aligned}
\label{eq:nonlinear_Lambda_r_step1}
\end{equation}
Note that $\Lambda$ is self-adjoint and $\Lambda^{r+2}u^{n,(2)}$ is well-defined due to the smoothing effect of $\SLt$.
According to Lemma~\ref{ode},
\[
    |w(s)|_r+|w^3(s)|_r
    \le
    C_{\Omega,d,r}
    \left(
        |u^{n,(1)}|_r+1
    \right).
\]
Using the contraction property of the linear half-step $\SLt$, we further have
\[
    |u^{n,(1)}|_r
    =
    \left|
    e^{\frac{\tau}{2}\varepsilon^2\Delta}u^n
    \right|_r
    \le
    |u^n|_r
    \le
    R_r(\varepsilon)
\]
and
\begin{equation}
    |w(s)|_r+|w^3(s)|_r
    \le
    C_{\Omega,d,r}^{(1)}(R_r(\varepsilon)+1).
    \label{eq:C3_Lambda_r}
\end{equation}
Again by Lemma~\ref{ode},
\[
    |u^{n,(2)}|_{r+2}
    =
    \left|
    \mathcal S_{\mathcal N}(\tau)u^{n,(1)}
    \right|_{r+2}
    \le
    C_{\Omega,d,r}^{(2)}
    \left(
        |u^{n,(1)}|_{r+2}+1
    \right).
\]
Substituting this estimate and \eqref{eq:C3_Lambda_r} into
\eqref{eq:nonlinear_Lambda_r_step1}, we get
\[
\begin{aligned}
& |u^{n,(2)}|_{r+1}^2
  -
  |u^{n,(1)}|_{r+1}^2
  +
  |u^{n,(2)}-u^{n,(1)}|_{r+1}^2      \\
  \le{}&
  C_{\Omega,d,r}^{(3)}\tau(R_r(\varepsilon)+1)
  \left(
      |u^{n,(1)}|_{r+2}+1
  \right).
\end{aligned}
\]
Using Young's inequality with the small parameter \(\varepsilon^2\), and
using \(0<\varepsilon\le1\), we obtain
\begin{equation}
\begin{aligned}
& |u^{n,(2)}|_{r+1}^2
  -
  |u^{n,(1)}|_{r+1}^2
  +
  |u^{n,(2)}-u^{n,(1)}|_{r+1}^2 \\
  \le{}&
  \varepsilon^2\tau |u^{n,(1)}|_{r+2}^2
  + C_{\Omega,d,r}^{(4)}
  \frac{(R_r(\varepsilon)+1)^2}{\varepsilon^2}\tau .
\end{aligned}
\label{eq:nonlinear_Lambda_r_step2}
\end{equation}
For the last linear half-step $\SLt =\exp\left(\frac{\tau}{2}\varepsilon^2\Delta\right)$, we keep the dissipative term. Since
\[
    {u^{n+1}-u^{n,(2)}}
    =
    \left(
    \mathcal I-e^{-\frac{\tau}{2}\varepsilon^2\Delta}
    \right)u^{n+1},
\]
applying \(\Lambda^{r+1}\) to both sides of the above  and
taking the \(L^2\)-inner product with \(2\Lambda^{r+1}u^{n+1}\) gives
\begin{equation}
\begin{aligned}
    |u^{n+1}|_{r+1}^2
    -
    |u^{n,(2)}|_{r+1}^2
    +
    |u^{n+1}-u^{n,(2)}|_{r+1}^2
    \le
    -\tau\varepsilon^2 |u^{n+1}|_{r+2}^2 .
\end{aligned}
\label{eq:last_linear_Lambda_r}
\end{equation}
Combining \eqref{eq:first_linear_Lambda_r},
\eqref{eq:nonlinear_Lambda_r_step2}, and
\eqref{eq:last_linear_Lambda_r}, we have
\begin{equation}
    |u^{n+1}|_{r+1}^2
    -
    |u^n|_{r+1}^2
    \le
    -\varepsilon^2\tau |u^{n+1}|_{r+2}^2
    +
    C_{\Omega,d,r}^{(4)}
    \frac{(R_r+1)^2}{\varepsilon^2}\tau .
    \label{eq:combined_Lambda_r_simplified}
\end{equation}

We now convert the higher-order dissipation in the above inequality into a damping term for
\(|u^{n+1}|_{r+1}\). By the Sobolev interpolation inequality,
\[
    |u^{n+1}|_{r+1}^2
    \le
    |u^{n+1}|_r
    |u^{n+1}|_{r+2}.
\]
Since \(|u^{n+1}|_r\le R_r(\varepsilon)\), we have
\[
    |u^{n+1}|_{r+2}^2
    \ge
    \frac{
        |u^{n+1}|_{r+1}^4
    }{
        (R_r(\varepsilon)+1)^2
    }.
\]
Then \eqref{eq:combined_Lambda_r_simplified} yields
\begin{equation}
    |u^{n+1}|_{r+1}^2-|u^n|_{r+1}^2
    \le
    -
    \frac{\varepsilon^2}{(R_r(\varepsilon)+1)^2}
    \tau |u^{n+1}|_{r+1}^4
    +
    C_{\Omega,d,r}^{(4)}
    \frac{(R_r(\varepsilon)+1)^2}{\varepsilon^2}\tau .
    \label{eq:nonlinear_damping_recursion_Lambda_r}
\end{equation}
We can then claim that
\begin{equation}
    |u^n|_{r+1}^2
    \le
    \max\left\{
        |u^0|_{r+1}^2,\,
        (C_{\Omega,d,r}^{(4)})^{\frac12}
        \frac{(R_r(\varepsilon)+1)^2}{\varepsilon^2}
    \right\},
    \qquad n\ge0,
    \label{eq:barrier_Lambda_r}
\end{equation}
which gives
\begin{equation}
    R_{r+1}(\varepsilon)
    \le
    \max\left\{|u^0|_{r+1},
    C_{\Omega,d,r}
    \frac{R_r(\varepsilon)+1}{\varepsilon}\right\}.
    \label{eq:Rr_plus_one_recursion_Lambda}
\end{equation}
The proof of \eqref{eq:R1_explicit_recursive} is complete.

Assume that, for some integer \(r\ge1\), the following hypothesis holds:
\begin{equation}
    R_r(\varepsilon)
    \le
    K_{r,u^0}
    \left(
        \varepsilon^{-r}+1
    \right).
    \label{eq:induction_hypothesis_Lambda_r}
\end{equation}
Note that the case of $r=1$ has been proved.
Then, from \eqref{eq:R1_explicit_recursive}, we get
\[
\begin{aligned}
    R_{r+1}
    &\le
    |u^0|_{r+1}
    +
    C_{\Omega,d,r}
    \frac{
        K_{r,u^0}(\varepsilon^{-r}+1)+1
    }{\varepsilon} \le
    K_{r+1,u^0}
    \left(
        \varepsilon^{-(r+1)}+1
    \right),
\end{aligned}
\]
where \(K_{r+1,u^0}>0\) depends only on \(\Omega\), \(d\), \(r\), and the initial energy $E^0$, independent of \(n\), \(T\), \(\tau\), and
\(\varepsilon\). 
The proof of \eqref{eq:explicit_Lambda_eps_bound} is complete.
\end{proof}

\section{Error estimate with optimal dependence on $\varepsilon$}\label{sec:error}
Blanes et al.  establish in \cite{blanes2025splitting} an error estimate for operator splitting methods, which covers the second-order Strang splitting scheme. However, the constant in this estimate is expressed only in terms of stability and regularity bounds, and its dependence on the PDE parameters is not explicit. We now provide the following error analysis with error bounds explicitly dependent on  $\varepsilon$.

\begin{lemma}[Local error in $H^r$ with explicit dependence on $\varepsilon$]
\label{lem:local-error-eps2}
Let $r\ge0$ be an integer, $0<\varepsilon\le1$, and $0<\tau\le1$.
Consider the Allen--Cahn equation~\eqref{eq:AC} with periodic boundary conditions. Assume that $u^0\in H^{r+4}(\Omega)$ is periodic and
$\|u^0\|_{L^\infty}\le1$.
For $t_n=n\tau$, define the local error of the Strang splitting method by
\begin{equation}
\label{eq:eps2-local-definition}
\delta^{n+1}
:=\mathcal S_{\mathcal L}\left(\frac{\tau}{2}\right)
\mathcal S_{\mathcal N}(\tau)
\mathcal S_{\mathcal L}\left(\frac{\tau}{2}\right)u(t_n)
-u(t_{n+1}).
\end{equation}
Then, for a given time $T>0$, there exists a constant $C_{r,T,u^0}>0$, depending only on $\Omega$, $d$, $r$, $T$, $\|u^0\|_{H^{r+4}}$, and independent of $n$, $\tau$, $\varepsilon$, such that
\begin{equation}
\label{eq:eps2-local-error}
\|\delta^{n+1}\|_{r,\Lambda}
\le C_{r,T,u^0}\varepsilon^2\tau^3,
\qquad n\tau\le T.
\end{equation}
In particular, for $r=0$,
\[
\|\delta^{n+1}\|_{L^2}
\le C_{0,T,u^0}\varepsilon^2\tau^3,
\qquad n\tau\le T.
\]
\end{lemma}

\begin{proof}
Set
\[
\mathcal L(w):=\varepsilon^2\Delta w,
\qquad
\mathcal N(w):=w-w^3,
\qquad
\mathcal H:=\mathcal L+\mathcal N.
\]
Throughout the following argument, $C_{r,T,u^0}>0$ denotes a generic constant depending only on $\Omega,\,d,\,r,\,T$, $\|u^0\|_{H^{r+4}}$, while \(C_{\Omega,d,m}>0\) denotes a generic constant depending only on $\Omega,\,d,\,m$.

We first derive regularity estimates independent of
$\varepsilon$. Let $m=r+4$. By Lemma~\ref{lem:HGN}, for every
$w\in H^m(\Omega)$ satisfying $\|w\|_{L^\infty}\le1$, we have
\[
\|w^3\|_{H^m}\le C_{\Omega,d,m}^{(1)}\|w\|_{H^m}.
\]
Applying $D^\alpha$ to the Allen--Cahn equation, taking the
$L^2$-inner product with $D^\alpha u$, and summing over
$|\alpha|\le m$, the periodic boundary condition gives
\begin{align*}
\frac12\frac{\mathrm d}{\mathrm dt}\|u(t)\|_{H^m}^2
&=-\varepsilon^2\sum_{|\alpha|\le m}
\|\nabla D^\alpha u(t)\|_{L^2}^2
+\sum_{|\alpha|\le m}
\left\langle D^\alpha(u-u^3),D^\alpha u\right\rangle \\
&\le
\bigl(\|u(t)\|_{H^m}+\|u(t)^3\|_{H^m}\bigr)
\|u(t)\|_{H^m}
\le C_{\Omega,d,m}^{(2)}\|u(t)\|_{H^m}^2.
\end{align*}
Here we have used the maximum principle
$\|u(t)\|_{L^\infty}\le1$.
Consequently, Gronwall's inequality yields
\begin{equation}
\label{eq:eps2-exact-regularity}
\|u(t)\|_{H^m}
\le e^{C_{\Omega,d,m}^{(3)}t}\|u^0\|_{H^m},
\qquad 0\le t\le T+1.
\end{equation}
The same calculation for the nonlinear subproblem, together with the
contractivity of the linear flow, gives
\begin{equation}
\label{eq:eps2-subflow-regularity}
\|\mathcal S_{\mathcal N}(s)w\|_{H^m}
\le e^{C_{\Omega,d,m}^{(4)}s}\|w\|_{H^m},
\qquad
\|\mathcal S_{\mathcal L}(s)w\|_{H^m}\le\|w\|_{H^m},
\qquad s\ge0,
\end{equation}
provided that $\|w\|_{L^\infty}\le1$.
Both subflows preserve this $L^\infty$ bound. Thus, all intermediate
states $w$ used below satisfy
\begin{equation}
\label{eq:eps2-intermediate-regularity}
\|w\|_{L^\infty}\le1,
\qquad
\|w\|_{H^{r+4}}\le C_{r,T,u^0}^{(1)}.
\end{equation}
We use the commutator convention
\[
[\mathcal F,\mathcal G](w)
:=\mathcal F'(w)\mathcal G(w)
-\mathcal G'(w)\mathcal F(w).
\]
Since
\[
\mathcal L'(w)\eta=\varepsilon^2\Delta\eta,
\qquad
\mathcal N'(w)\eta=(1-3w^2)\eta,
\]
a direct calculation gives
\begin{align}
[\mathcal L,\mathcal N](w)
&=-6\varepsilon^2w|\nabla w|^2,
\label{eq:eps2-first-commutator}\\
[\mathcal L,[\mathcal L,\mathcal N]](w)
&=-24\varepsilon^4D^2w(\nabla w,\nabla w)
-12\varepsilon^4w|D^2w|_F^2,
\label{eq:eps2-linear-commutator}\\
[\mathcal N,[\mathcal N,\mathcal L]](w)
&=12\varepsilon^2w(2w^2-1)|\nabla w|^2,
\label{eq:eps2-nonlinear-commutator}
\end{align}
where
\[
D^2w(\nabla w,\nabla w)
=\sum_{i,j=1}^d\partial_{ij}w\,\partial_iw\,\partial_jw,
\qquad
|D^2w|_F^2=\sum_{i,j=1}^d|\partial_{ij}w|^2.
\]
The products in \eqref{eq:eps2-linear-commutator} have total
derivative order four, whereas those in
\eqref{eq:eps2-nonlinear-commutator} have total derivative order two.
Applying Lemma~\ref{lem:HGN} and \eqref{eq:eps2-intermediate-regularity}, we obtain
\begin{align}
\|[\mathcal L,[\mathcal L,\mathcal N]](w)\|_{H^r}
&\le C_{\Omega,d,r}^{(5)}\varepsilon^4\|w\|_{H^{r+4}}
\le C_{r,T,u^0}^{(2)}\varepsilon^4,
\label{eq:eps2-linear-commutator-bound}\\
\|[\mathcal N,[\mathcal N,\mathcal L]](w)\|_{H^r}
&\le C_{\Omega,d,r}^{(6)}\varepsilon^2\|w\|_{H^{r+2}}
\le C_{r,T,u^0}^{(3)}\varepsilon^2.
\label{eq:eps2-nonlinear-commutator-bound}
\end{align}
We also record two estimates used below. Since $d\le3$, the Sobolev
embedding $H^{r+2}(\Omega)\hookrightarrow W^{r,\infty}(\Omega)$,
the Leibniz rule, and \eqref{eq:eps2-intermediate-regularity} imply
\begin{equation}
\label{eq:eps2-nonlinear-lipschitz}
\begin{aligned}
\|\mathcal N(p)-\mathcal N(q)\|_{H^r}
&=\|(1-p^2-pq-q^2)(p-q)\|_{H^r}\\
&\le C_{\Omega,d,r}^{(7)}
\|1-p^2-pq-q^2\|_{W^{r,\infty}}\|p-q\|_{H^r}\\
&\le C_{r,T,u^0}^{(4)}\|p-q\|_{H^r},
\end{aligned}
\end{equation}
whenever $p$ and $q$ satisfy the bounds in
\eqref{eq:eps2-intermediate-regularity}.
Recall that
\[
\mathcal S_{\mathcal N}(s)w=g_s(w),
\qquad
g_s(\zeta)=\frac{e^s\zeta}
{\sqrt{1+(e^{2s}-1)\zeta^2}}.
\]
The derivatives of $g_s$ are uniformly bounded for
$0\le s\le1$ and $|\zeta|\le1$. The chain rule therefore gives
\begin{equation}
\label{eq:eps2-linearized-nonlinear-flow}
\|g_s'(w)\|_{W^{r,\infty}}\le C_{r,T,u^0}^{(5)},
\qquad
\|g_s'(w)\eta\|_{H^r}\le C_{r,T,u^0}^{(6)}\|\eta\|_{H^r},
\qquad 0\le s\le1.
\end{equation}
For simplicity of notations, set
\[
z(t):=\mathcal S_{\mathcal L}\left(\frac t2\right)
\mathcal S_{\mathcal N}(t)
\mathcal S_{\mathcal L}\left(\frac t2\right)u(t_n),
\qquad
\mathcal R(t):=z'(t)-\mathcal H(z(t)),
\qquad 0\le t\le\tau.
\]
For each fixed $t$, write
\[
w(s):=\mathcal S_{\mathcal N}(s)\mathcal S_{\mathcal L}\left(\frac t2\right)u(t_n),
\qquad 0\le s\le t.
\]
The dependence of $w(s)$ on the fixed parameter $t$ is suppressed
for simplicity of notation. Differentiating $z(t)$, including the
dependence of the first linear half-step on $t$, gives
\begin{equation}
\label{eq:eps2-defect-decomposition}
\begin{aligned}
\mathcal R(t)
={}& \mathcal S_{\mathcal L}\left(\frac t2\right)\mathcal N(w(t))
-\mathcal N\left(\mathcal S_{\mathcal L}\left(\frac t2\right)w(t)\right)\\
&+\frac12 \mathcal S_{\mathcal L}\left(\frac t2\right)
\left[g_t'(w(0))\mathcal L(w(0))-\mathcal L(w(t))\right].
\end{aligned}
\end{equation}
We first consider the first line of
\eqref{eq:eps2-defect-decomposition}. Differentiating
$\mathcal S_{\mathcal L}(t/2-s)\mathcal N(\mathcal S_{\mathcal L}(s)w(t))$
with respect to $s$ yields
$-\mathcal S_{\mathcal L}(t/2-s)[\mathcal L,\mathcal N]
(\mathcal S_{\mathcal L}(s)w(t))$.
Applying the same calculation to $[\mathcal L,\mathcal N]$
and integrating twice, we obtain
\begin{equation}
\label{eq:eps2-linear-transport}
\begin{aligned}
&\mathcal S_{\mathcal L}\left(\frac t2\right)\mathcal N(w(t))
-\mathcal N\left(\mathcal S_{\mathcal L}\left(\frac t2\right)w(t)\right)\\
={}&\frac t2 \mathcal S_{\mathcal L}\left(\frac t2\right)
[\mathcal L,\mathcal N](w(t))
-
\int_0^{t/2}\left(\frac t2-s\right)
\mathcal S_{\mathcal L}\left(\frac t2-s\right)
[\mathcal L,[\mathcal L,\mathcal N]]
\bigl(\mathcal S_{\mathcal L}(s)w(t)\bigr)\,\mathrm ds.
\end{aligned}
\end{equation}
For the second line of \eqref{eq:eps2-defect-decomposition}, the function
$g_s'(w(0))\mathcal L(w(0))-\mathcal L(w(s))$
satisfies a linear differential equation with coefficient
$\mathcal N'(w(s))$, zero initial value, and source
$-[\mathcal L,\mathcal N](w(s))$.
The variation-of-constants formula gives
\[
g_t'(w(0))\mathcal L(w(0))-\mathcal L(w(t))
=-\int_0^t g_{t-s}'(w(s))
[\mathcal L,\mathcal N](w(s))\,\mathrm ds.
\]
Moreover, differentiating along the nonlinear flow gives
\[
\frac{\mathrm d}{\mathrm ds}
\left[g_{t-s}'(w(s))[\mathcal L,\mathcal N](w(s))\right]
=g_{t-s}'(w(s))
[\mathcal N,[\mathcal N,\mathcal L]](w(s)).
\]
Since $g_0'=1$, integrating this identity and exchanging the order of
integration yield
\begin{equation}
\label{eq:eps2-nonlinear-transport}
\begin{aligned}
&g_t'(w(0))\mathcal L(w(0))-\mathcal L(w(t))\\
=&-t[\mathcal L,\mathcal N](w(t))
+\int_0^t s\,g_{t-s}'(w(s))
[\mathcal N,[\mathcal N,\mathcal L]](w(s))\,\mathrm ds.
\end{aligned}
\end{equation}
Substituting \eqref{eq:eps2-linear-transport} and
\eqref{eq:eps2-nonlinear-transport} into
\eqref{eq:eps2-defect-decomposition}, the two terms containing
$[\mathcal L,\mathcal N](w(t))$ cancel. Hence,
\begin{equation}
\label{eq:eps2-defect-representation}
\begin{aligned}
\mathcal R(t)
={}&-\int_0^{t/2}\left(\frac t2-s\right)
\mathcal S_{\mathcal L}\left(\frac t2-s\right)
[\mathcal L,[\mathcal L,\mathcal N]]
\bigl(\mathcal S_{\mathcal L}(s)w(t)\bigr)\,\mathrm ds\\
&+\frac12 \mathcal S_{\mathcal L}\left(\frac t2\right)
\int_0^t s\,g_{t-s}'(w(s))
[\mathcal N,[\mathcal N,\mathcal L]](w(s))\,\mathrm ds.
\end{aligned}
\end{equation}
Combining \eqref{eq:eps2-linear-commutator-bound},
\eqref{eq:eps2-nonlinear-commutator-bound},
\eqref{eq:eps2-linearized-nonlinear-flow}, and the contractivity of
$\mathcal S_{\mathcal L}$, we derive
\begin{equation}
\label{eq:eps2-defect-bound}
\begin{aligned}
\|\mathcal R(t)\|_{H^r}
&\le C_{r,T,u^0}^{(7)}\varepsilon^4
\int_0^{t/2}\left(\frac t2-s\right)\,\mathrm ds
+C_{r,T,u^0}^{(8)}\varepsilon^2\int_0^t s\,\mathrm ds\\
&\le C_{r,T,u^0}^{(9)}(\varepsilon^4+\varepsilon^2)t^2
\le C_{r,T,u^0}^{(10)}\varepsilon^2t^2,
\qquad 0\le t\le\tau.
\end{aligned}
\end{equation}
Finally, $z(0)=u(t_n)$, and the variation-of-constants formula gives
\begin{align*}
z(t)-u(t_n+t)
=\int_0^t \mathcal S_{\mathcal L}(t-s)
\bigl[\mathcal N(z(s))-\mathcal N(u(t_n+s))
+\mathcal R(s)\bigr]\,\mathrm ds.
\end{align*}
Using \eqref{eq:eps2-nonlinear-lipschitz},
\eqref{eq:eps2-defect-bound}, and Gronwall's inequality, we obtain
\begin{align*}
\|z(t)-u(t_n+t)\|_{H^r}
&\le C_{r,T,u^0}^{(11)}\varepsilon^2
\int_0^t e^{C_{r,T,u^0}^{(12)}(t-s)}s^2\,\mathrm ds\\
&\le C_{r,T,u^0}^{(13)}\varepsilon^2t^3,
\qquad 0\le t\le\tau\le1.
\end{align*}
Taking $t=\tau$ and applying Lemma~\ref{lem:sobolev_equivalence} prove
\eqref{eq:eps2-local-error}.

\end{proof}

\begin{theorem}
\label{thm:global-error-eps2}
Consider the Allen--Cahn equation~\eqref{eq:AC} with the periodic boundary condition and the dimension $d\leq 3$. Let $r\ge0$ be an integer. Assume that $u^0\in H^{r+4}(\Omega)\cap W^{4,\infty}(\Omega)$ is periodic, $\|u^0\|_{L^\infty}\le1$, $0<\varepsilon\le1$, and $0<\tau\le1$.
Then, for a given time $T>0$, there exists a constant $C>0$,
depending only on $\Omega,\,d,\,r,\,T$, and $\|u^0\|_{H^{r+4}}$, such that
\begin{equation}
\label{eq:eps2-global-error}
\sup_{n\tau\le T}\|u^n-u(t_n)\|_{H^r}
\le C\varepsilon^2\tau^2.
\end{equation}
Moreover, in the case of $r=0$, we have the following $L^\infty$-norm
error estimate:
\begin{equation}
\label{eq:eps2-global-error-infinity}
\sup_{n\tau\le T}\|u^n-u(t_n)\|_{L^\infty}
\le\widetilde C\varepsilon^2\tau^2,
\end{equation}
where $\widetilde C$ depends only on $\Omega,\,d,\,T$, and
$\|u^0\|_{W^{4,\infty}}$.
\end{theorem}

\begin{proof}
Throughout the following argument, $C_{r,T,u^0}>0$ denotes a generic
constant depending only on $\Omega,\,d,\,r,\,T$, $\|u^0\|_{H^{r+4}}$, while \(C_{\Omega,d,m}>0\) denotes a generic constant depending only on $\Omega,\,d,\,m$.

We first establish the regularity bound for the numerical solution. Let $m=r+4$. The discrete maximum principle gives $\|u^n\|_{L^\infty}\le1$ for all $n\ge0$.
By \eqref{eq:eps2-subflow-regularity},
\[
\|u^{n+1}\|_{H^m}
\le e^{C_{\Omega,d,m}^{(1)}\tau}\|u^n\|_{H^m}.
\]
Iteration yields
\begin{equation}
\label{eq:eps2-numerical-regularity}
\|u^n\|_{H^m}
\le e^{C_{\Omega,d,m}^{(2)}n\tau}\|u^0\|_{H^m},
\qquad
\sup_{n\tau\le T+1}\|u^n\|_{H^{r+4}}
\le C_{r,T,u^0}^{(1)}.
\end{equation}
Together with \eqref{eq:eps2-exact-regularity} and
\eqref{eq:eps2-subflow-regularity}, this estimate gives the same
$\varepsilon$-independent regularity bounds for the exact and numerical
intermediate states used below.
By the triangle inequality,
\begin{equation}
\label{eq:eps2-error-decomposition}
\begin{aligned}
\|u^{n+1}-u(t_{n+1})\|_{H^r}
\le{}&\|\mathcal S(\tau)u^n-\mathcal S(\tau)u(t_n)\|_{H^r}\\
&+\|\mathcal S(\tau)u(t_n)-u(t_{n+1})\|_{H^r}.
\end{aligned}
\end{equation}
Using Lemma~\ref{lem:sobolev_equivalence} and Lemma~\ref{lem:local-error-eps2}, we obtain
\begin{equation}
\label{eq:eps2-consistency-for-global}
\|\mathcal S(\tau)u(t_n)-u(t_{n+1})\|_{H^r}
\le C_{r,T,u^0}^{(2)}\varepsilon^2\tau^3,
\qquad n\tau\le T.
\end{equation}

Let
\[
v^n:=\mathcal S_{\mathcal L}\left(\frac\tau2\right)u^n,
\qquad
w^n:=\mathcal S_{\mathcal L}\left(\frac\tau2\right)u(t_n).
\]
The maximum principle, \eqref{eq:eps2-exact-regularity},
\eqref{eq:eps2-subflow-regularity}, and
\eqref{eq:eps2-numerical-regularity} imply
\[
\max\{\|v^n\|_{L^\infty},\|w^n\|_{L^\infty}\}\le1,
\qquad
\|v^n\|_{H^{r+4}}+\|w^n\|_{H^{r+4}}
\le C_{r,T,u^0}^{(3)}.
\]
The same bounds hold, up to a change in the constant, for
$\mathcal S_{\mathcal N}(s)v^n$ and $\mathcal S_{\mathcal N}(s)w^n$, $0\le s\le\tau$.
Using the integral formulation of the nonlinear subproblem and
\eqref{eq:eps2-nonlinear-lipschitz}, we have
\begin{align*}
&\|\mathcal S_{\mathcal N}(s)v^n-\mathcal S_{\mathcal N}(s)w^n\|_{H^r}\\
\le&\|v^n-w^n\|_{H^r}
+C_{r,T,u^0}^{(4)}\int_0^s
\|\mathcal S_{\mathcal N}(\sigma)v^n-\mathcal S_{\mathcal N}(\sigma)w^n\|_{H^r}
\,\mathrm d\sigma.
\end{align*}
Gronwall's inequality therefore gives
\[
\|\mathcal S_{\mathcal N}(\tau)v^n-\mathcal S_{\mathcal N}(\tau)w^n\|_{H^r}
\le e^{C_{r,T,u^0}^{(5)}\tau}\|v^n-w^n\|_{H^r}.
\]
By the contractivity of the two linear half-steps,
\begin{equation}
\label{eq:eps2-splitting-stability}
\|\mathcal S(\tau)u^n-\mathcal S(\tau)u(t_n)\|_{H^r}
\le e^{C_{r,T,u^0}^{(6)}\tau}\|u^n-u(t_n)\|_{H^r}.
\end{equation}
Combining \eqref{eq:eps2-error-decomposition},
\eqref{eq:eps2-consistency-for-global}, and
\eqref{eq:eps2-splitting-stability}, we derive
\begin{equation}
\label{eq:eps2-error-recursion}
\|u^{n+1}-u(t_{n+1})\|_{H^r}
\le e^{C_{r,T,u^0}^{(7)}\tau}\|u^n-u(t_n)\|_{H^r}
+C_{r,T,u^0}^{(8)}\varepsilon^2\tau^3.
\end{equation}
Since $u^0=u(0)$, iteration of \eqref{eq:eps2-error-recursion} gives
\begin{align*}
\|u^n-u(t_n)\|_{H^r}
&\le C_{r,T,u^0}^{(9)}\varepsilon^2\tau^3
\sum_{j=0}^{n-1}e^{C_{r,T,u^0}^{(10)}j\tau}\\
&\le C_{r,T,u^0}^{(11)}\varepsilon^2\tau^3 n
e^{C_{r,T,u^0}^{(12)}T}
\le C\varepsilon^2\tau^2,
\qquad n\tau\le T.
\end{align*}
Taking the supremum over $n\tau\le T$ proves
\eqref{eq:eps2-global-error}.

Finally, we prove \eqref{eq:eps2-global-error-infinity}.
In the remainder of the proof, $C_{T,u^0}>0$ denotes a generic constant depending only on $\Omega,d,T$, $\|u^0\|_{W^{4,\infty}}$, and independent of $n,\tau$, $\varepsilon$.
Applying \eqref{eq:eps2-exact-regularity} with $m=4$, using
$H^4(\Omega)\hookrightarrow W^{2,\infty}(\Omega)$ for $d\le3$, and
noting that $\|u^0\|_{H^4}\le C_\Omega\|u^0\|_{W^{4,\infty}}$,
we find that all intermediate states in
\eqref{eq:eps2-defect-representation} satisfy
\[
\|w\|_{W^{2,\infty}}\le C_{T,u^0}^{(1)}.
\]
The explicit commutator formulas
\eqref{eq:eps2-linear-commutator}--\eqref{eq:eps2-nonlinear-commutator}
therefore imply
\[
\|[\mathcal L,[\mathcal L,\mathcal N]](w)\|_{L^\infty}
\le C_{T,u^0}^{(2)}\varepsilon^4,
\qquad
\|[\mathcal N,[\mathcal N,\mathcal L]](w)\|_{L^\infty}
\le C_{T,u^0}^{(3)}\varepsilon^2.
\]
Moreover, $\sup_{\zeta\in\mathbb R}|g_s'(\zeta)|\le e^s$.
Using \eqref{eq:eps2-defect-representation} and the
$L^\infty$-contractivity of $\mathcal S_{\mathcal L}$, we obtain
\[
\|\mathcal R(t)\|_{L^\infty}
\le C_{T,u^0}^{(4)}\varepsilon^2t^2,
\qquad 0\le t\le\tau.
\]
Since $\mathcal N$ is Lipschitz on $[-1,1]$, the same
variation-of-constants argument as in the local error proof gives
\[
\|\delta^{n+1}\|_{L^\infty}
\le C_{T,u^0}^{(5)}\varepsilon^2\tau^3,
\qquad n\tau\le T.
\]
The pointwise estimate for $g_\tau'$ and the contractivity of
$\mathcal S_{\mathcal L}$ also give
\[
\|\mathcal S(\tau)u^n-\mathcal S(\tau)u(t_n)\|_{L^\infty}
\le e^\tau\|u^n-u(t_n)\|_{L^\infty}.
\]
Consequently,
\[
\|u^{n+1}-u(t_{n+1})\|_{L^\infty}
\le e^\tau\|u^n-u(t_n)\|_{L^\infty}
+C_{T,u^0}^{(6)}\varepsilon^2\tau^3.
\]
Iteration, together with $u^0=u(0)$, yields
\[
\|u^n-u(t_n)\|_{L^\infty}
\le C_{T,u^0}^{(7)}\varepsilon^2\tau^3
\sum_{j=0}^{n-1}e^{j\tau}
\le\widetilde C\varepsilon^2\tau^2,
\qquad n\tau\le T.
\]
This proves \eqref{eq:eps2-global-error-infinity} and completes the proof.
\end{proof}

\begin{remark}
Previous work \cite{quan2026stability} established stability and error estimates of the Strang splitting method with nonuniform time meshes under homogeneous Neumann boundary conditions, without providing explicit dependence on $\varepsilon$. Due to the lack of an energy dissipation law in the case of nonuniform time meshes, it is unknown how to establish error estimates with optimal dependence on $\varepsilon$.
\end{remark}

\section{Numerical Experiments}\label{sec:numerical}
In this section, we compare the theoretical predictions with
numerical results for the Strang splitting method~\eqref{eq:strang}.
We examine the second-order temporal convergence, the dependence
of the temporal errors on $\varepsilon$, the maximum principle, and modified-energy dissipation.
We also investigate long-time Sobolev behavior and compare the
observed \(\varepsilon\)-dependence with the theoretical stability bounds.

The discrete Sobolev seminorm $ |\cdot|_{r,N}$ is defined by
\begin{equation}
 |u^n|_{r,N}^2=\sum_{j=0}^{N-1}\lambda_j^r|\widehat u_j^n|^2,
 \label{eq:discrete-seminorm}
\end{equation}
where $\lambda_j$ is the eigenvalue of $-\Delta$ given in  \eqref{def:eigenvalue} and $\widehat u_j^n$ are Fourier or cosine expansion coefficients of $u^n$. Here, $N$ denotes the total number of grid points.

\subsection{Temporal accuracy and long-time Sobolev bounds}

We consider the Allen--Cahn equation \eqref{eq:AC} in
$\Omega=(0,2\pi)^2$ with periodic boundary conditions and
\[
\begin{aligned}
u^0(x_1,x_2)=\cos(x_1)\cos(x_2).
\end{aligned}
\]
The initial condition is smooth, periodic, and independent of
$\varepsilon$. For periodic boundary conditions, we employ the
Fourier spectral method.

We first investigate the temporal accuracy with $\varepsilon=0.1$.
We use a $512\times512$ spatial mesh and take different time steps
\[
 \tau=\frac1{32},\ \frac1{64},\ \frac1{128},\
      \frac1{256},\ \frac1{512}.
\]
The reference solution is computed on the same spatial mesh with
a maximum time step $\tau_{\rm ref}=10^{-5}$.
Table~\ref{tab:periodic-accuracy} reports $L^2$, $H^1$, and $H^2$ errors defined in \eqref{def:seminorm} at $T=1$.
The observed rates are close to two, demonstrating the second-order temporal accuracy in this test.

\begin{table}[h!]
  \centering
  \caption{Temporal errors and convergence rates at $T=1$ under periodic boundary conditions.}
  \label{tab:periodic-accuracy}
  \input{example2_periodic_accuracy_table}
\end{table}

We next test the $\varepsilon$-dependence of the temporal errors.
Take a $512\times512$ spatial mesh, $\tau=0.01$, $T=1$, and
\[
  \varepsilon= 0.4,\ 0.2,\ 0.1,\ 0.05,\ 0.025,\ 0.0125,
\]
respectively. The same initial condition is used for all values
of $\varepsilon$. For each $\varepsilon$, the reference solution
is computed with
\[
  \tau_{\rm ref}=\frac{\tau}{64}=1.5625\times10^{-4}.
\]
We compute numerically the $H^r$ errors as follows:
\[
  E_{H^r}(\varepsilon)
  =\max_{0\le n\tau\le T}\bigl\|u^n-u_{\rm ref}(t_n)\bigr\|_{H^r},
  \qquad r=0,1,2.
\]
Figure~\ref{fig:periodic-error-scaling} plots
$\log E_{H^r}(\varepsilon)$ against $-\log\varepsilon$
for $r=0,1,2$.
For the smaller values of $\varepsilon$, the error curves become
approximately parallel to the reference lines of slope $-2$.
This behavior shows close agreement with the $\varepsilon^2$
scaling in the error bound of Theorem~\ref{thm:global-error-eps2}.

\begin{figure}[h!]
  \centering
  \includegraphics[width=0.98\textwidth]
    {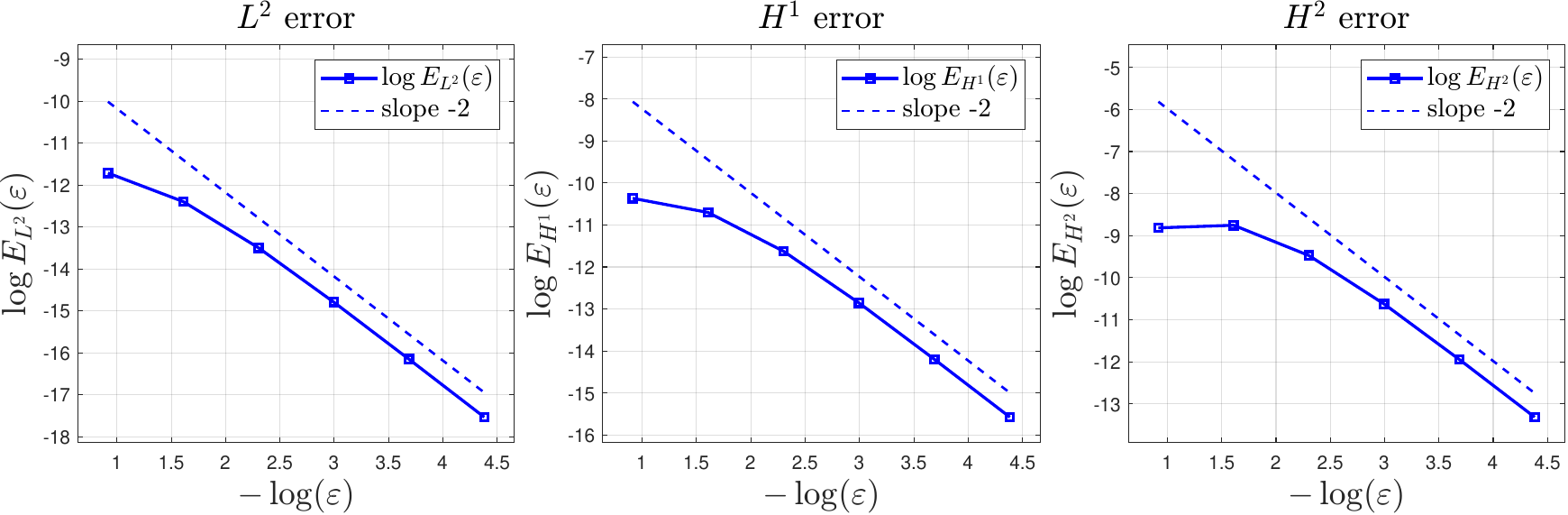}
  \caption{$\log E_{H^r}(\varepsilon)$ against
  $-\log\varepsilon$ for $r=0,1,2$ for the Allen--Cahn equation
  with periodic boundary conditions.
  The dashed lines indicate reference slopes $-2$.}
  \label{fig:periodic-error-scaling}
\end{figure}

We further investigate the long-time behavior of the
Strang splitting method. Take a $512\times512$ spatial mesh,
$\varepsilon=0.1$, $\tau=0.01$, and $T=500$.
Figure~\ref{fig:periodic-long-time-2d} illustrates the instantaneous seminorms $|u^n|_{r,N}$, $r=1,2,3$,
with respect to time.
It can be observed that these seminorms remain bounded
over the simulated time interval, providing numerical evidence consistent with the uniform-in-time Sobolev boundedness.

\begin{figure}[h!]
  \centering
  \includegraphics[width=0.4\textwidth]
    {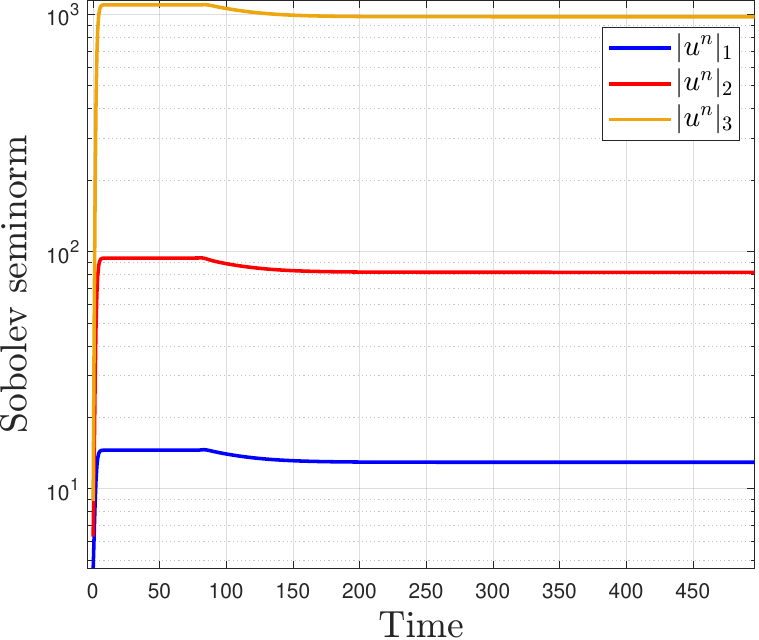}
  \caption{Long-time evolution of $|u^n|_{r,N}$,
  $r=1,2,3$, for the Allen--Cahn equation
  with periodic boundary conditions.}
  \label{fig:periodic-long-time-2d}
\end{figure}

\subsection{Maximum principle, energy dissipation, and long-time Sobolev bounds}

We use a seven-circle phase-field profile on $\Omega=(0,2\pi)^2$ as the initial condition of the Allen--Cahn equation with periodic boundary condition:
\[
 u^0(x,y)=0.98\tanh\left(
 -\frac{d(x,y)}{0.05\sqrt{2}}\right),
\]
where
$$
\begin{aligned}
d(x,y)&=\min_{1\leq j\leq7}
 \left(\sqrt{(x-x_j)^2+(y-y_j)^2}-\rho_j\right),\\
(x_j)_{j=1}^7 &=\left(\frac\pi2,\frac\pi4,\frac\pi2,\pi, \frac{3\pi}2,\pi,\frac{3\pi}2\right), \\ 
(y_j)_{j=1}^7 &=\left(\frac\pi2,\frac{3\pi}4,\frac{5\pi}4, \frac\pi4,\frac\pi4,\pi,\frac{3\pi}2\right),\\ 
(\rho_j)_{j=1}^7 &=\left(\frac\pi5,\frac{2\pi}{15},\frac{2\pi}{15},\frac\pi{10},\frac\pi{10},\frac\pi4,\frac\pi4\right).
\end{aligned}
$$

The Fourier pseudospectral method is employed for spatial discretization.
For the maximum-principle and energy dissipation tests, we take a $512\times 512$ spatial mesh, $\varepsilon=0.1$, $\tau=0.01$, and $T=50$.
Figure~\ref{fig:max-energy} shows that the maximum norm of the numerical solution remains bounded by one, while both the original Allen--Cahn energy and the modified energy $\widetilde E^n$ defined in \eqref{eq:energy} decay in time.

\begin{figure}[h!]
 \centering
 \includegraphics[width=0.7\textwidth]{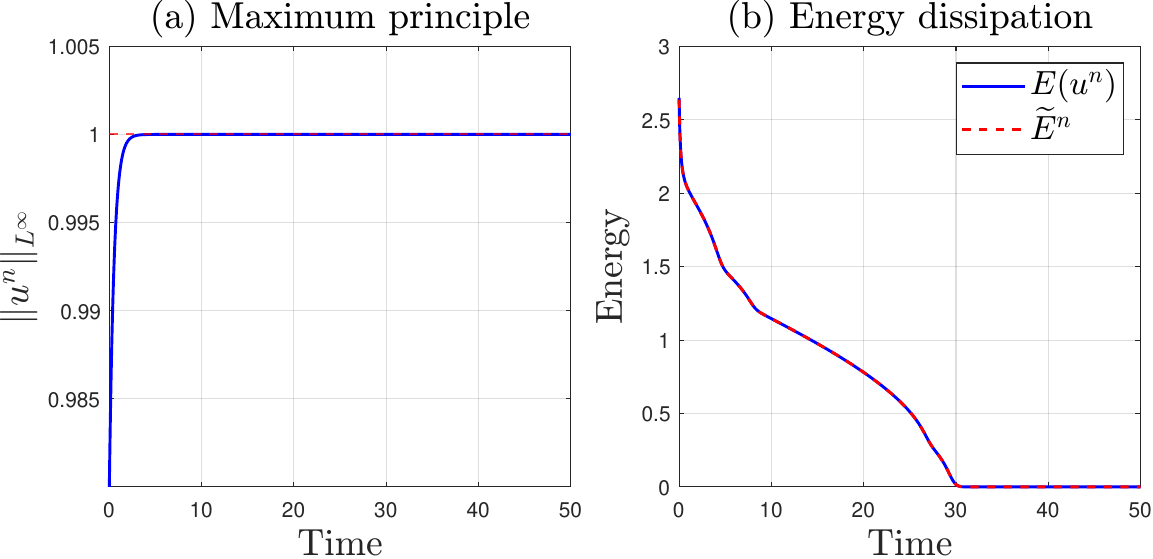}
 \caption{Maximum principle and energy dissipation of the Strang splitting method for the Allen--Cahn equation with periodic boundary condition and seven-circle initial condition.}
 \label{fig:max-energy}
\end{figure}

\begin{figure}[h!]
 \centering
 \includegraphics[width=0.7\textwidth]{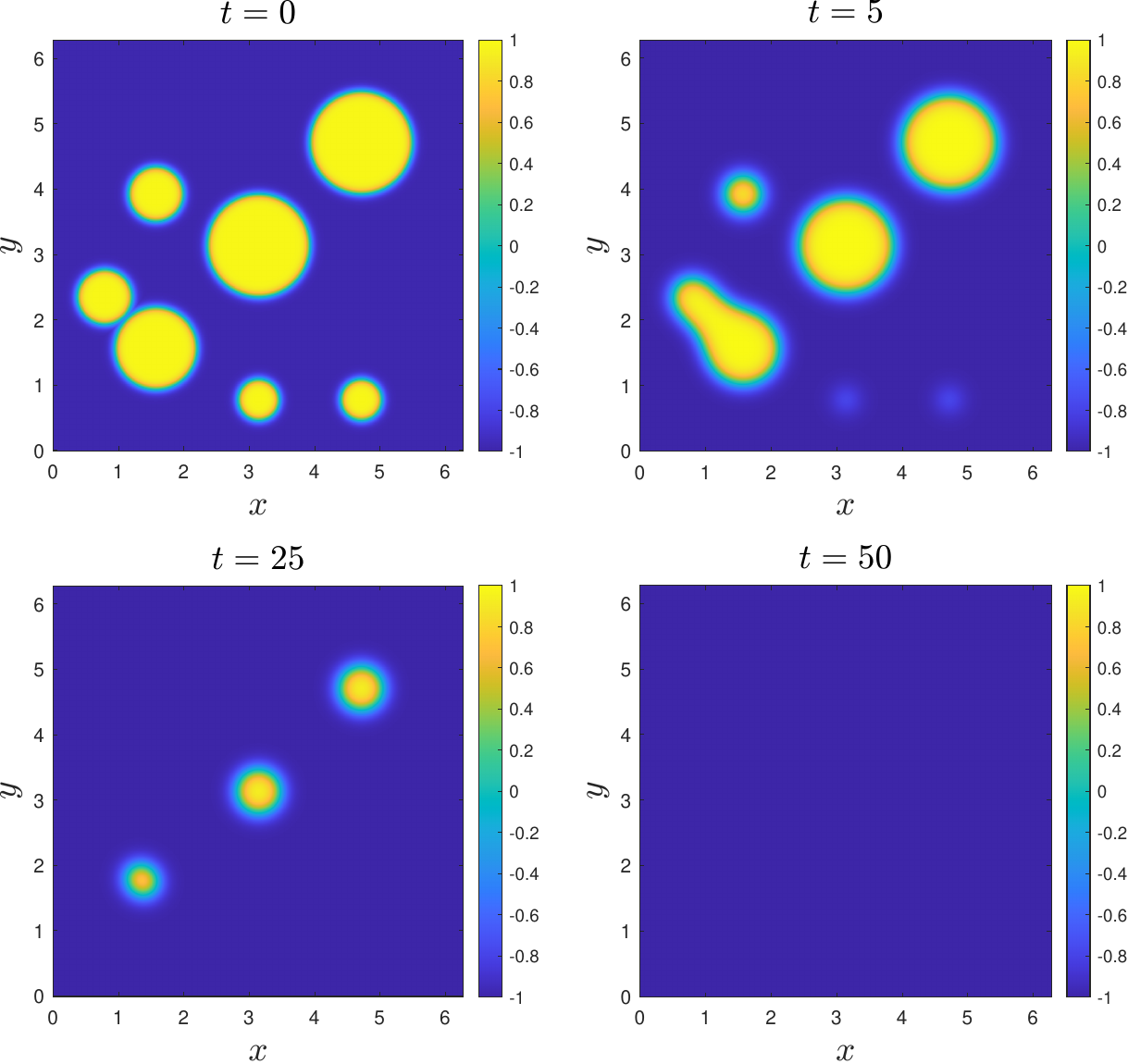}
 \caption{Seven-circle phase-field evolution at $t=0,5,25,$ and $50$. It can be observed that small components disappear and neighboring components merge.}
 \label{fig:seven-circle}
\end{figure}

We now test the $\varepsilon$-dependence of $R_r(\varepsilon)$.
Take a \(512\times512\) spatial mesh, $\tau=0.01$, $T=50$, and
\(
 \varepsilon= 0.05,\ 0.025,\ 0.0125,\ 0.00625,
\)
respectively. The truncation number in \eqref{eq:discrete-seminorm} is $N= 512^2$ and we compute numerically $R_r(\varepsilon)$.
Figure~\ref{fig:epsilon-scaling2} plots
$\log R_r(\varepsilon)$ against $-\log\varepsilon$ for $r=1,2,3$.
It can be observed that the seminorms of $\log R_r(\varepsilon)$ are approximately linear w.r.t. $-\log\varepsilon$, with slopes less than the upper bounds in \eqref{eq:explicit_Lambda_eps_bound}.

\begin{figure}[h!]
 \centering
\includegraphics[width=0.98\textwidth]{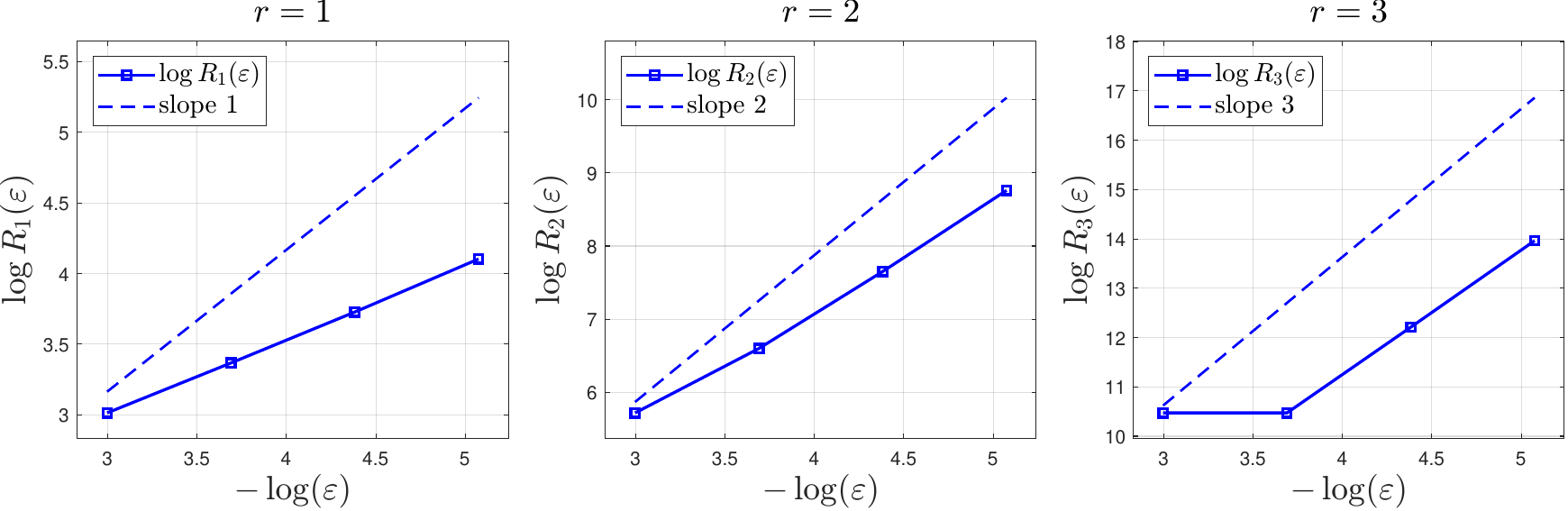}
 \caption{$\log R_r(\varepsilon)$ against $-\log\varepsilon$ for $r=1,2,3$  for the Allen--Cahn equation with periodic boundary conditions.}
 \label{fig:epsilon-scaling2}
\end{figure}

Take the $512\times 512$ mesh as before,
$\varepsilon=0.1$, $\tau=0.01$, and $T=500$.
Figure~\ref{fig:long-time-2d} illustrates instantaneous seminorms $|u^n|_{r,N}$, $r=1,2,3$ w.r.t. the time. It can be observed that these seminorms remain bounded over long time, validating the uniform-in-time boundedness.

\begin{figure}[h!]
 \centering
 \includegraphics[width=0.4\textwidth]{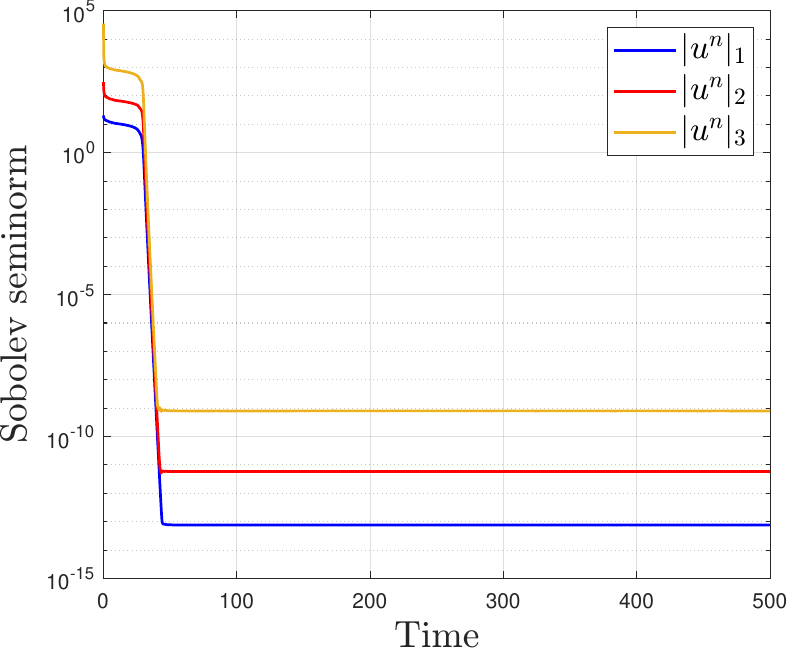}
 \caption{Long-time evolution of $|u^n|_{r,N}$, $r=1,2,3$ for the Allen--Cahn equation with periodic boundary conditions.}
 \label{fig:long-time-2d}
\end{figure}

\section{Conclusion}\label{sec:conclusion}

In this work, we establish long-time Sobolev stability and error estimates with optimal $\varepsilon$-dependence for the Strang splitting method applied to the AC equation with periodic boundary conditions. We prove that the numerical solution admits uniform $H^r$-seminorm bounds $\mathcal O(\varepsilon^{-r}+1)$, which is consistent with the Sobolev bound of the exact solution. In addition, \(H^r\)-norm error estimate is established to be $\mathcal O(\varepsilon^2 \tau^2)$ on any prescribed finite time interval, which improves existing works with polynormial dependence on $\varepsilon^{-1}$. Numerical experiments verify temporal accuracy, maximum principle, energy dissipation law, uniform-in-time boundedness, and \(\varepsilon\)-dependent Sobolev error estimate. 

Error estimates that explicitly track the dependence on \(\varepsilon\) have also been investigated extensively for other equations, such as the epitaxial thin-film equation \cite{LiQiao2016,qiao2017error} and the Cahn--Hilliard equation \cite{guo2021improved, guo2024refined}. A natural direction for future work is to improve the dependence on $\varepsilon^{-1}$ for numerical methods of different gradient flows.


\section*{Acknowledgments}
 
C. Quan is supported by National Natural Science Foundation of
China (Grant No. 12622124), Guangdong Basic and Applied Basic Research Foundation (Grant
No. 2023B1515020030), and Shenzhen Science and Technology Innovation Program (Grant No.
JCYJ20230807092402004).
Z. Tan is supported by the National Nature Science Foundation of China
(12371418), Guangdong Natural Science Foundation (2024A1515010694, 2022A1515010426), and Guangdong Province Key Laboratory of Computational Science at the Sun Yat-sen University (2020B1212060032).

\appendix
\section{Proof of Lemma~\ref{lem:longtime-bound}}
\label{append:1}

Throughout the proof, \(C>0\) denotes a generic constant which may change
from line to line. Unless otherwise stated, the constants are independent
of \(t\) and of the final time \(T\). The parameter \(\varepsilon>0\) is fixed.

We first prove the first-order Sobolev bound. By the energy dissipation, we have
\begin{equation}
    \frac{\varepsilon^2}{2}\|\nabla u(t)\|_{L^2}^2
    \le
    E(u^0),
    \qquad \forall t\ge0.
    \label{eq:H1_energy_appendix}
\end{equation}
and therefore,
\begin{equation}
    |u(t)|_1
    =
    \|\nabla u(t)\|_{L^2}
    \le
    \frac{\sqrt{2E(u^0)}}{\varepsilon} \leq C_1 (\varepsilon^{-1}+1),
    \quad \forall t\ge0,
    \label{eq:H1_bound_appendix}
\end{equation}
where $C_1$ depends only on \(\Omega\), \(d\), and \(|u^0|_1\).

Assume now that, for some \(m\ge2\), the estimate has been proved at
order \(m-1\), namely,
\begin{equation}
    |u(t)|_{m-1}
    \le
    C_{m-1}\left(\varepsilon^{-(m-1)}+1\right),
    \qquad \forall t\ge0.
    \label{eq:induction_hypothesis_exact_appendix}
\end{equation}
Applying \(\Lambda^m\) to the Allen--Cahn equation \eqref{eq:AC}
and taking the \(L^2\)-inner product with \(\Lambda^m u\), we obtain
\begin{equation}
\begin{aligned}
    \frac12\frac{\rm d}{{\rm d}t}|u(t)|_m^2
    &=
    \varepsilon^2
    \left\langle
        \Lambda^m\Delta u,
        \Lambda^m u
    \right\rangle
    +
    \left\langle
        \Lambda^m(u-u^3),
        \Lambda^m u
    \right\rangle                                      \\
    &=
    -\varepsilon^2 |u(t)|_{m+1}^2
    +
    \left\langle
        \Lambda^m(u-u^3),
        \Lambda^m u
    \right\rangle .
\end{aligned}
\label{eq:m_energy_start_appendix}
\end{equation}
By the self-adjointness of \(\Lambda\), we have
\[
\begin{aligned}
    \left\langle
        \Lambda^m(u-u^3),
        \Lambda^m u
    \right\rangle
    &=
    \left\langle
        \Lambda^{m-1}(u-u^3),
        \Lambda^{m+1}u
    \right\rangle                                      \\
    &\le
    \left(
        |u(t)|_{m-1}
        +
        |u^3(t)|_{m-1}
    \right)
    |u(t)|_{m+1}.
\end{aligned}
\label{eq:nonlinear_m_appendix}
\]
By the norm equivalence in Lemma~\ref{lem:sobolev_equivalence} and
and the maximum principle of $u$, we have
\begin{equation}
    |u^3(t)|_{m-1}
    \le
    C_{\Omega,d,m}^{(1)}
    \left(
        |u(t)|_{m-1}+1
    \right)
    \label{eq:cubic_moser_appendix}
\end{equation}
and consequently,
\begin{equation}
\begin{aligned}
    \left\langle
        \Lambda^m(u-u^3),
        \Lambda^m u
    \right\rangle
    &\le
    C_{\Omega,d,m}^{(2)}
    \left(
        |u(t)|_{m-1}+1
    \right)
    |u(t)|_{m+1}                                      \\
    &\le
    \frac{\varepsilon^2}{2}|u(t)|_{m+1}^2
    +
    \frac{C_{\Omega,d,m}}{2\varepsilon^2}
    \left(
        |u(t)|_{m-1}+1
    \right)^2,
\end{aligned}
\label{eq:young_m_appendix}
\end{equation}
where $C_{\Omega,d,m}^{(1)},~C_{\Omega,d,m}^{(2)}$, and $C_{\Omega,d,m}$ are postive constants depend only on $\Omega,~d,$ and $m$.

Substituting \eqref{eq:young_m_appendix} into
\eqref{eq:m_energy_start_appendix}, we get
\begin{equation}
    \frac{\rm d}{{\rm d} t}|u(t)|_m^2
    +
    \varepsilon^2 |u(t)|_{m+1}^2
    \le
    \frac{C_{\Omega,d,m}}{\varepsilon^2}
    \left(
        |u(t)|_{m-1}+1
    \right)^2 \leq \frac{C_{\Omega,d,m}}{\varepsilon^2}
    B_{m-1}^2,
    \quad \forall t\ge0
    \label{eq:m_energy_ineq_appendix}
\end{equation}
with
\[
    B_{m-1}
    :=
    C_{m-1}\left(\varepsilon^{-(m-1)}+1\right)+1.
\]
The Sobolev interpolation inequality says that 
\[
\begin{aligned}
    |u(t)|_m^2
    &=
    \sum_{j\ge1}\lambda_j^m |u_j(t)|^2 
    \le
    \left(
        \sum_{j\ge1}\lambda_j^{m-1}|u_j(t)|^2
    \right)^{1/2}
    \left(
        \sum_{j\ge1}\lambda_j^{m+1}|u_j(t)|^2
    \right)^{1/2} \\
    & =
    |u(t)|_{m-1}|u(t)|_{m+1}
\end{aligned}
\]
that implies from \eqref{eq:induction_hypothesis_exact_appendix}
\[
    |u(t)|_m^2
    \le
    B_{m-1}|u(t)|_{m+1},
\]
i.e.,
\begin{equation}
    |u(t)|_{m+1}^2
    \ge
    \frac{|u(t)|_m^4}{B_{m-1}^2}.
    \label{eq:interpolation_damping_appendix}
\end{equation}
Substituting \eqref{eq:interpolation_damping_appendix} into
\eqref{eq:m_energy_ineq_appendix}, 
we obtain the differential inequality
\begin{equation}
    \frac{\rm d}{{\rm d}t}|u(t)|_{m}^2
    +
    \frac{\varepsilon^2}{B_{m-1}^2}|u(t)|_{m}^4
    \le
    \frac{C_{\Omega,d,m}}{\varepsilon^2}B_{m-1}^2 .
    \label{eq:riccati_m_appendix}
\end{equation}
By the comparison principle, we have 
\[
    |u(t)|_m^2
    \le
    \max\left\{
        |u^0|_m^2,
        C_{\Omega,d,m}^{\frac12}
        \frac{B_{m-1}^2}{\varepsilon^2}
    \right\}.
\]
Therefore,
\begin{equation}
    |u(t)|_m
    \le
    C_m
    \left(
        \varepsilon^{-m}+1
    \right),
    \quad \forall t\ge0 .
    \label{eq:m_final_bound_with_epsilon_appendix}
\end{equation}
Here \(C_m>0\) is independent of \(t\), \(T\), and \(\varepsilon\), and
depends only on \(\Omega\), \(d\), \(m\), and \(|u^0|_m\). This closes the induction and completes the proof of Lemma~\ref{lem:longtime-bound}.

\bibliographystyle{siamplain}
\bibliography{references}
\end{document}

%% file: example2_periodic_accuracy_table.tex
\begin{tabular}{rrrrrrr}
\toprule
$\tau$ & $L^2$ error & rate & $H^1$ error & rate & $H^2$ error & rate \\
\midrule
$\frac{1}{32}$ & 1.3418e-05 & -- & 8.7135e-05 & -- & 7.5043e-04 & -- \\
$\frac{1}{64}$ & 3.3545e-06 & 2.0000 & 2.1786e-05 & 1.9999 & 1.8764e-04 & 1.9997 \\
$\frac{1}{128}$ & 8.3865e-07 & 2.0000 & 5.4465e-06 & 2.0000 & 4.6913e-05 & 1.9999 \\
$\frac{1}{256}$ & 2.0967e-07 & 1.9999 & 1.3616e-06 & 2.0000 & 1.1728e-05 & 2.0000 \\
$\frac{1}{512}$ & 5.2431e-08 & 1.9997 & 3.4042e-07 & 2.0000 & 2.9321e-06 & 2.0000 \\
\bottomrule
\end{tabular}